\documentclass[10pt]{amsart}
\usepackage[utf8]{inputenc}
\usepackage{amsmath,amssymb,amsthm}
\usepackage{hyperref}
\usepackage{float}
\DeclareMathOperator{\polylog}{polylog}

\newcommand{\eps}{\varepsilon}
\newcommand{\R}{\mathbb{R}}
\newcommand{\E}{\mathbb{E}}
\newcommand{\Var}{\mathrm{Var}}
\newcommand{\Cov}{\mathrm{Cov}}

\newcommand{\floor}[1]{\left\lfloor #1 \right\rfloor}
\newtheorem{definition}{Definition}
\usepackage{enumitem}
\usepackage{parskip}

\usepackage{amsmath,amssymb,amsthm}
\usepackage{graphicx}
\usepackage{hyperref}
\usepackage{mathrsfs}
\usepackage{thmtools}

\newcommand{\StructureExtra}{} 
\usepackage{xcolor}
\definecolor{BLUE}{rgb}{0,0,1}

\usepackage{tikz}
\usepackage{tikz-cd}
\title[]{The Structure of Extremal Bad Science matrices}
\author[]{Shridhar Sinha}
\email{ssinha19@uw.edu}
\address{University of Washington, Seattle, WA 98195, USA}
\theoremstyle{plain}

\newtheorem{theorem}{Theorem}[section]
\newtheorem{lemma}{Lemma}[section]
\newtheorem{corollary}{Corollary}[section]
\newtheorem{proposition}{Proposition}[section]

\newtheorem{conjecture}{Conjecture}[section]
\newtheorem*{theorem*}{Theorem}
\newtheorem*{lemma*}{Lemma}
\newtheorem*{corollary*}{Corollary}
\newtheorem*{proposition*}{Proposition}

\begin{document}

\begin{abstract}
We study the \emph{bad science matrix problem}: among all matrices \(A\in\mathbb{R}^{n\times n}\) whose rows have unit \(\ell_2\)-norm, determine the maximum of
\[
\beta(A)=\frac{1}{2^n}\sum_{x\in\{\pm1\}^n}\|Ax\|_\infty.
\]
Steinerberger~\cite{Steinerberger} showed that the optimal asymptotic rate is \((1+o(1))\sqrt{2\log n}\), and that this rate is attained with high probability by matrices with i.i.d.\ \(\pm1\) entries, after normalization. More recent explicit constructions~\cite{ExplicitConstructions} achieve \(\beta(A)\ge\sqrt{\log_2(n)+1}\), which lies within a constant factor of the asymptotic optimum.
In this paper we bridge the gap between the probabilistic and explicit approaches. We give a geometric description of extremizers as (nearly) isoperimetrically extremal partitions of the \(n\)-dimensional hypercube induced by the rows of \(A\). We obtain precise rates for heuristic constructions by recasting the maximization of \(\beta(A)\) in the language of high-dimensional central-limit theorems as in~\cite{FangKoikeLiuZhao2023}. Using these connections, we present a family of explicit deterministic matrices\(\ A_n\) that exist for all $n$ under the assumption of Hadamard’s conjecture, and for infinitely many $n$ unconditionally, such that for all $n$ sufficiently large
$$\beta(A_n)\ge\bigl(1 - \frac{\log\log(2n)}{4\log(2n)}\bigr)\sqrt{2\log{2n}}.$$
\end{abstract}
\maketitle
\section{Introduction}
\subsection{Introduction}
The \emph{bad science matrix problem} models a “dishonest” testing scenario, where a researcher runs many fair statistical tests on random data in hopes of finding at least one atypical result. Concretely, let $A\in\R^{n\times n}$ be a matrix whose rows $a_i\in\R^n$ are normalized in $\ell_2$, $\|a_i\|_2=1$.  We consider the quantity 
$$
\beta(A)\;=\;\frac{1}{2^n}\sum_{x\in\{-1,1\}^n}\|A x\|_{\infty},
$$
which is equivalently the expectation $\E[\|Ax\|_\infty]$ for a random Rademacher vector $x\in\{-1,1\}^n$. Here $\|A x\|_\infty=\max_{1\le i\le n}|\langle a_i,x\rangle|$ is the largest absolute value among the linear tests $a_i$ applied to the data vector $x$.  In statistical terms, each row of $A$ represents a (fair) linear test of the $n$ random signs in $x$, and $\beta(A)$ measures the average largest test statistic over all $2^n$ possible outcomes.  The problem is to understand how large $\beta(A)$ can be under the unit‐norm constraint on the rows. 
This setup has a natural interpretation both in hypothesis testing and in geometry.  From the testing viewpoint, a \emph{bad scientist} pre-selects many unit‐norm test directions and then looks at a sequence of coin flips.  Even though the coin is fair, by chance, one of these tests will often yield a surprisingly large value, yielding an (incorrect) claim of significance. If any test is unusually large, the researcher obtains a small (but spurious) $p$-value and (incorrectly) concludes that the fair coin is biased. In geometric terms, the matrix $A$ maps the vertices of the discrete hypercube $\{-1,1\}^n$ into $\R^n$; one then asks whether a typical image point has a coordinate that is \emph{significantly} larger than average.  Steinerberger observes that such matrices correspond to affine images of the cube whose points “on average [have] at least one large coordinate”.  Equivalently, one can view $\beta(A)$ as measuring how far the hypercube can be “rotated” or embedded so that its vertices tend to lie outside the smaller cubes in $\ell_\infty$-norm.  This dual perspective connects to classical discrepancy and vector‐balancing problems (for instance, to the Komlós conjecture in discrepancy theory), but the bad science problem is a distinct “functional-balancing” variant.  In summary, the bad science matrix problem captures the risk of false positives when many fair tests are run, and it raises fundamental questions about how a bounded-norm linear map can distort the discrete cube’s geometry.  

\subsection{Existing Results} 
Steinerberger {\cite{Steinerberger}} established the first asymptotic bounds for this problem.  He proved that, as $n\to\infty$, the maximum possible value of $\beta(A)$ (over all $n\times n$ matrices with unit-$\ell_2$ rows) grows like 
$$
\max_{\|a_i\|_2=1}\beta(A)\;=\;(1+o(1))\sqrt{2\log n}\;.
$$
This result shows that the worst-case average sup-norm is of order $\sqrt{2\log n}$.  Moreover, the proof shows that this rate is attained (up to lower-order terms) by a random matrix with independent $\pm1/\sqrt{n}$ entries. In other words, a matrix with i.i.d.\ Rademacher rows (properly scaled) typically achieves $\beta(A) = (1+o(1))\sqrt{2\log n}$, matching the theoretical maximum order given by Gaussian‐maxima heuristics.   
However, all conjectured and verified extremal matrices in lower dimensions (until $n=8$) are highly structured and of low rank, very unlike the random asymptotic extremizers. An example for $n=5$ is
\[A \;=\; \frac{1}{2\sqrt{3}}
\begin{pmatrix}
2 & 2 & 0 & 0 & 2 \\
-2 & 2 & 0 & 2 & 0 \\
-2 & 0 & 0 & -2 & 2 \\
0 & -\sqrt{3} & \sqrt{3} & \sqrt{3} & \sqrt{3} \\
0 & \sqrt{3} & \sqrt{3} & -\sqrt{3} & -\sqrt{3}
\end{pmatrix}.\]

Building on this, Albors, Bhatti, Ganjoo, Guo, Kunisky, Mukherjee, Stepin, and Zeng {\cite{ExplicitConstructions}} provided explicit constructions and structural results.  They exhibit concrete $n\times n$ matrices $A$ achieving 
$$
\beta(A)\;\ge\;\sqrt{\log_2(n+1)}\;,
$$
improving upon trivial bounds and coming within a constant factor of the $\sqrt{2\log n}$ rate.  Their construction uses combinatorial designs (e.g. Hadamard‐type and tree‐based constructions) to ensure that $\|Ax\|_\infty$ is large for many corners $x$ of the cube.  In addition, the authors of \cite{ExplicitConstructions} prove remarkable structure theorems for extremal matrices: every entry of an optimal bad science matrix must be the square root of a rational number.  Using these insights, they completely solve the problem for small dimensions, determining exact maximizing matrices for $n\le4$.  These results highlight the geometry of extremal examples and show that while random matrices are asymptotically optimal, the true maximizers in lower dimensions exhibit rich algebraic structure.

\section{Main results}

\subsection{Overview}
We summarize here the principal results of the paper and their principal consequences; precise statements and quantitative refinements appear in Theorems~\ref{thm:Fourier Characterization}--\ref{thm:asymptotics}.

\medskip

\noindent\textit{(1) Fourier characterization.}  Lemma~\ref{thm:Fourier Characterization} and Theorem~\ref{lem:Structure of Extremal matrices} provides a reduction of the optimization problem for
\[
\beta(A)=\frac{1}{2^n}\sum_{x\in\{\pm1\}^n}\|Ax\|_\infty
\]
in terms of Fourier Analysis on the indicator functions of the set of vertices in the hypercube closest (in the Euclidean sense) to each row of \(A\).  Concretely, \(\beta(A)\) is controlled by and, in the extremal setting, asymptotically approaches an explicit functional of Fourier coefficients of these sets; this reduction converts the original extremal problem into a problem about distributing Fourier mass on the cube, and gives us natural heuristic candidates for extremal matrices.

\noindent\textit{(2) Structural stability of extremizers.}  Theorem~\ref{lem:Structure of Extremal matrices} implies that any sequence of matrices whose \(\beta\)-values attain the asymptotic maximum induces a Voronoi tessellation with strong regularity properties: the vector of cell volumes converges in \(\ell^2\) to the constant \(1/(2n)\), the row vectors agree with their normalized cell centroids up to vanishing error, and, for all but \(o(n)\) indices, the cells are asymptotically optimal for Level-1 Fourier weight. These quantitative, pointwise refinements together give a rigid geometric characterization of all asymptotic extremizers.

\noindent\textit{(3) Constructions and sharp asymptotics.} Given that the exact optimization problem seems hard, Theorem~\ref{thm:asymptotics} establishes precise rates for asymptotically optimal families — normalized random sign matrices as introduced in \cite{Steinerberger} and explicit constructions based on Hadamard matrices, which we introduce. Using tools from \cite{FangKoikeLiuZhao2023}, we show that, under certain light and heuristically sound assumptions, these families have optimal first order asymptotics. We also provide numerical evidence that our constructions have \(\beta\) greater than the ones presented in \cite{Steinerberger}, thus providing the best known asymptotic constructions from a numerical point of view.

We also use the theory developed here to provide a new interpretation of previous results in \cite{ExplicitConstructions}. This shows why the natural generalization of the algebraically structured low-dimensional extremizers falls short of the optimal rate asymptotically.  

\label{sec:main-results}
\subsection{Setup.}

Throughout the paper, for an $n\times n$ real matrix $A$ with rows $a_1,\dots,a_n$ satisfying $\|a_i\|_2=1$ we write
\[
\beta(A)=\frac{1}{2^n}\sum_{x\in\{-1,1\}^n}\|Ax\|_\infty.
\]

For each row index $i$ define the full cell
\[
C_i(A):=\{x\in\{-1,1\}^n:\ | \langle a_i,x\rangle|=\max_{1\le j\le n}|\langle a_j,x\rangle|\},
\]
and the positive half of the cell
\[
S_i(A):=\{x\in C_i(A):\ \langle a_i,x\rangle\ge 0\}.
\]
Thus $C_i(A)=S_i(A)\cup(-S_i(A))$ for every $i$.  When there are no hypercube vertices with same inner product with two or more rows, the sets $\{S_i(A),-S_i(A):1\le i\le n\}$ form a partition of the hypercube $\{-1,1\}^n$, a fact that is central to our proofs.  To avoid repetition from this point on we will therefore assume, without loss of generality, that $A$ has no ties; Proposition~\ref{prop:no-ties} shows this assumption is generic (and harmless) and that results for tie-free matrices extend to all matrices by a small perturbation/continuity argument.

\begin{tikzpicture}[scale=1.4, every node/.style={font=\small}]

  \def\ox{0.5}  
  \def\oy{0.25} 

  \newcommand{\Coord}[4]{\coordinate (#1) at ({#2+\ox*#4},{#3+\oy*#4});}

  \Coord{A}{-1}{-1}{-1} 
  \Coord{B}{ 1}{-1}{-1} 
  \Coord{C}{ 1}{ 1}{-1} 
  \Coord{D}{-1}{ 1}{-1} 
  \Coord{E}{-1}{-1}{ 1} 
  \Coord{F}{ 1}{-1}{ 1} 
  \Coord{G}{ 1}{ 1}{ 1} 
  \Coord{H}{-1}{ 1}{ 1} 

  \draw[gray!80] (A) -- (B) -- (C) -- (D) -- cycle; 
  \draw[gray!80] (E) -- (F) -- (G) -- (H) -- cycle;           
  \draw[gray!80] (A) -- (E);
  \draw[gray!80] (B) -- (F);
  \draw[gray!80] (C) -- (G);
  \draw[gray!80] (D) -- (H);

  \fill[red]  (G) circle (2.6pt) node[above right=2pt] {$S_1$};
  \fill[red]  (A) circle (2.6pt) node[below left=2pt] {$-S_1$};

  \fill[blue] (F) circle (2.6pt) node[right=2pt] {$S_2$};
  \fill[blue] (D) circle (2.6pt) node[left=2pt] {$-S_2$};

  \fill[green!60!black] (E) circle (2.6pt) node[left=2pt] {$S_3$};
  \fill[green!60!black] (H) circle (2.6pt) node[above left=2pt] {$S_3$};
  \fill[green!60!black] (B) circle (2.6pt) node[below right=2pt] {$-S_3$};
  \fill[green!60!black] (C) circle (2.6pt) node[above right=2pt] {$-S_3$};

  \node[anchor=west] at (2.2,0.1) {%
    \begin{minipage}{6.2cm}
      \centering\small
      \textbf{Induced partitions of $\{\pm1\}^3$}\\[4pt]
      for the $3\times 3$ matrix
      \[
        A=\begin{pmatrix}
        \tfrac{1}{\sqrt{3}} & \tfrac{1}{\sqrt{3}} & \tfrac{1}{\sqrt{3}}\\[6pt]
        \tfrac{1}{\sqrt{3}} & -\tfrac{1}{\sqrt{3}} & \tfrac{1}{\sqrt{3}}\\[6pt]
        -\tfrac{1}{\sqrt{2}} & 0 & \tfrac{1}{\sqrt{2}}
        \end{pmatrix}
      \]
      that was proven optimal in \cite{ExplicitConstructions}.
    \end{minipage}
  };
\end{tikzpicture}

\subsection{Structure of extremal bad science matrices.}\label{sec:structure}
The first result gives a concise Fourier–analytic upper bound for $\beta(A)$ and connects the optimization problem to Level-1 Fourier Weight of the indicator functions of the induced cells on hypercube. For \(f:\{-1,1\}^n \to \{0,1\}\), the Level-1 Fourier Weight is defined as
\[
W_1[f] \;=\; \sum_{i=1}^n \big( \mathbb{E}[\,f(x)\,x_i\,] \big)^2, \]
where the expectation is over all hypercube vectors \(x \in \{-1,1\}^n\).
\begin{restatable}[]{lemma}{FC}\label{thm:Fourier Characterization}
Let \( A \) be an \( n \times n \) matrix that has rows normalized in \(\ell_2\). Define the subset $S_i = \left\{ x \in \{-1,1\}^n: \|A x\|_\infty = \langle A_i, x \rangle \right\}.$
Then, the value of \( \beta(A) \) satisfies  
\[
\beta(A)  \le2\sum_{i=1}^{n} \sqrt{W_1[\mathbf{1}_{S_i}]}.
\]
\end{restatable}

An immediate, useful consequence (combining Lemma~\ref{thm:Fourier Characterization} with the Level–1 inequality of Talagrand) is a universal upper bound on $\beta(A)$ and the asymptotic tightness of the Fourier bound for extremal matrices.

\begin{restatable}[Fourier Characterization]{theorem}{Structure}\label{lem:Structure of Extremal matrices}
Let \( A \) be an \( n \times n \) matrix with rows normalized in \(\ell_2\) that has optimal beta value, then
\[
\beta(A) = 2(1+o(1))\sum_{i=1}^{n} \sqrt{W_1[\mathbf{1}_{S_i}]}.
\]
Furthermore, we have
\[
\sum_{i=1}^n\Bigl( \frac{|S_i|}{2^n} -\frac{1}{2n}\Bigr)^2
= \mathcal{O}\!\Biggl(\frac{\log\log(2n)}{\sqrt{\log(2n)}}\Biggr).
\]
\StructureExtra
\end{restatable}

In qualitative terms, the lemma shows that the partition \(\{S_i(A),-S_i(A)\}_{i=1}^n\) induced by an extremal matrix \(A\) in high dimensions is highly rigid and geometrically regular.  Concretely, the cells approach what is known in the literature \cite{DuGunzburgerJu2003} as a constrained centroidal Voronoi tessellation: the normalized centroid \(c_i\) of each cell \(S_i(A)\) coincides with the corresponding row \(a_i\) of \(A\) up to a vanishing error.  The \(\ell_2\)-convergence of the volume vector to \(\tfrac{1}{2n}\mathbf{1}\) expresses that the cells become asymptotically equal in size, so no single cell carries a large fraction of the mass.  Finally, except for \(o(n)\) exceptional indices, each cell is asymptotically optimal for Level-1 Fourier weight (given that most of the volumes are close to equal). Together, these features describe a partition that is simultaneously centroidal, equidistributed, and (nearly) Level-1 isoperimetrically extremal.

\begin{tikzpicture}[scale=1.4, every node/.style={font=\small}]\label{fig:subcubes}

  \def\ox{0.50}  
  \def\oy{0.25}  
  \def\px{-0.40} 
  \def\py{0.60}  

  \newcommand{\Coord}[5]{\coordinate (#1) at ({#2+\ox*#4+\px*#5},{#3+\oy*#4+\py*#5});}

  \Coord{v1111}{ 1}{ 1}{ 1}{ 1}
  \Coord{v111m}{ 1}{ 1}{ 1}{-1}
  \Coord{v11m1}{ 1}{ 1}{-1}{ 1}
  \Coord{v11mm}{ 1}{ 1}{-1}{-1}
  \Coord{v1m11}{ 1}{-1}{ 1}{ 1}
  \Coord{v1m1m}{ 1}{-1}{ 1}{-1}
  \Coord{v1mm1}{ 1}{-1}{-1}{ 1}
  \Coord{v1mmm}{ 1}{-1}{-1}{-1}
  \Coord{vm111}{-1}{ 1}{ 1}{ 1}
  \Coord{vm11m}{-1}{ 1}{ 1}{-1}
  \Coord{vm1m1}{-1}{ 1}{-1}{ 1}
  \Coord{vm1mm}{-1}{ 1}{-1}{-1}
  \Coord{vmm11}{-1}{-1}{ 1}{ 1}
  \Coord{vmm1m}{-1}{-1}{ 1}{-1}
  \Coord{vmmm1}{-1}{-1}{-1}{ 1}
  \Coord{vmmmm}{-1}{-1}{-1}{-1}

  \draw[gray!80] (vmmmm) -- (v1mmm) -- (v11mm) -- (vm1mm) -- cycle; 
  \draw[gray!80] (vmm1m) -- (v1m1m) -- (v111m) -- (vm11m) -- cycle; 
  \draw[gray!80] (vmmmm) -- (vmm1m);
  \draw[gray!80] (v1mmm) -- (v1m1m);
  \draw[gray!80] (v11mm) -- (v111m);
  \draw[gray!80] (vm1mm) -- (vm11m);

  \draw[gray!80] (vmmm1) -- (v1mm1) -- (v11m1) -- (vm1m1) -- cycle; 
  \draw[gray!80] (vmm11) -- (v1m11) -- (v1111) -- (vm111) -- cycle; 
  \draw[gray!80] (vmmm1) -- (vmm11);
  \draw[gray!80] (v1mm1) -- (v1m11);
  \draw[gray!80] (v11m1) -- (v1111);
  \draw[gray!80] (vm1m1) -- (vm111);

  \draw[gray!80] (vmmmm) -- (vmmm1);
  \draw[gray!80] (v1mmm) -- (v1mm1);
  \draw[gray!80] (v11mm) -- (v11m1);
  \draw[gray!80] (vm1mm) -- (vm1m1);
  \draw[gray!80] (vmm1m) -- (vmm11);
  \draw[gray!80] (v1m1m) -- (v1m11);
  \draw[gray!80] (v111m) -- (v1111);
  \draw[gray!80] (vm11m) -- (vm111);

  \fill[red]  (v1111) circle (2.6pt) node[above right=2pt] {$S_1$};
  \fill[red]  (v111m) circle (2.6pt) node[above right=2pt] {$S_1$};
  \fill[red]  (vmmm1) circle (2.6pt) node[below left=2pt] {$-S_1$};
  \fill[red]  (vmmmm) circle (2.6pt) node[below left=2pt] {$-S_1$};

  \fill[blue] (v11m1) circle (2.6pt) node[right=2pt] {$S_2$};
  \fill[blue] (v11mm) circle (2.6pt) node[right=2pt] {$S_2$};
  \fill[blue] (vmm11) circle (2.6pt) node[left=2pt] {$-S_2$};
  \fill[blue] (vmm1m) circle (2.6pt) node[left=2pt] {$-S_2$};

  \fill[green!60!black] (v1m11) circle (2.6pt) node[left=2pt] {$S_3$};
  \fill[green!60!black] (v1m1m) circle (2.6pt) node[left=2pt] {$S_3$};
  \fill[green!60!black] (vm1m1) circle (2.6pt) node[above left=2pt] {$-S_3$};
  \fill[green!60!black] (vm1mm) circle (2.6pt) node[above left=2pt] {$-S_3$};

  \fill[orange] (v1mm1) circle (2.6pt) node[below right=2pt] {$S_4$};
  \fill[orange] (v1mmm) circle (2.6pt) node[below right=2pt] {$S_4$};
  \fill[orange] (vm111) circle (2.6pt) node[above=2pt] {$-S_4$};
  \fill[orange] (vm11m) circle (2.6pt) node[above=2pt] {$-S_4$};

  \node[anchor=west] at (2.2,0.1) {%
    \begin{minipage}{6.2cm}
      \centering\small
      \textbf{Induced partitions of $\{\pm1\}^4$}\\[4pt]
      for the $4\times 4$ matrix
      \[
        A=\begin{pmatrix}
        \tfrac{1}{\sqrt{3}} & \tfrac{1}{\sqrt{3}} & \tfrac{1}{\sqrt{3}} & 0\\[6pt]
        \tfrac{1}{\sqrt{3}} & \tfrac{1}{\sqrt{3}} & -\tfrac{1}{\sqrt{3}} & 0\\[6pt]
        \tfrac{1}{\sqrt{3}} & -\tfrac{1}{\sqrt{3}} & \tfrac{1}{\sqrt{3}} & 0\\[6pt]
        \tfrac{1}{\sqrt{3}} & -\tfrac{1}{\sqrt{3}} & -\tfrac{1}{\sqrt{3}} & 0
        \end{pmatrix}
      \]
    that was proven optimal in \cite{ExplicitConstructions}.
    \end{minipage}
  };

\end{tikzpicture}

\subsection{Asymptotic constructions.}\label{sec:constructions}
Given that we do not know of a global optimum for large values of \(n\), we turn to simple constructions that match the upper bound asymptotically. The first proven example of such a construction is from Steinerberger's original paper \cite{Steinerberger}, which shows that an \(n\times n\) matrix with rows sampled uniformly at random from \(\{-1,1\}^n\) and then normalized, matches the upper bound asymptotically with high probability. The result leaves a little to be desired in terms of the precise asymptotics of the convergence, and whether we can do so deterministically. We resolve both these questions. The following theorem provides a precise asymptotic expansion for two concrete constructions that attain the optimal leading order: (i) normalized random sign matrices and (ii) deterministic orthonormal almost–Hadamard matrices obtained by truncating and orthonormalizing a Hadamard block. The expansion records the second–order term originating from classical Gaussian extreme–value theory and provides an explicit and negligible CLT error term.

\begin{restatable}[Orthonormal Almost–Hadamard matrix]{definition}{defOrth}\label{def:Orth}
For a fixed integer n, let \(m\ge n\) be the smallest integer such that a Hadamard matrix of order m exists, then an \(n\times n\) matrix \(Q\in\mathbb{R}^{n\times n}\) is called an \emph{orthonormal almost–Hadamard matrix} if:
\begin{enumerate}
  \item \(H\in\{-1,1\}^{m\times m}\) is a Hadamard matrix with \(H H^{\top}=mI_{m}\).
  \item \(U = H_{[1:n,\,1:n]}\big/\sqrt{m}\in\mathbb{R}^{n\times n}\) is the top–left \(n\times n\) block of \(H\).
  \item \(Q\) is obtained by the QR‐factorization \(U = Q\,R\), so that \(Q Q^{\top}=I_{n}\).
\end{enumerate}
\end{restatable}

We emphasize that this construction is completely explicit. For example we construct two non-trivial small examples

\begin{figure}[ht]
\centering
\[
{%
  \renewcommand{\arraystretch}{2}
  \begin{pmatrix}
    -\dfrac{1}{\sqrt{3}} & \dfrac{1}{\sqrt{6}} & -\dfrac{1}{\sqrt{2}} \\
    -\dfrac{1}{\sqrt{3}} & -\dfrac{2}{\sqrt{6}} & 0 \\
    -\dfrac{1}{\sqrt{3}} & \dfrac{1}{\sqrt{6}} & \dfrac{1}{\sqrt{2}}
  \end{pmatrix}
  \qquad
  \begin{pmatrix}
    -\dfrac{1}{\sqrt{5}} & \dfrac{2}{\sqrt{30}} & \dfrac{2}{\sqrt{42}} & -\dfrac{1}{\sqrt{14}} & -\dfrac{1}{\sqrt{2}} \\
    -\dfrac{1}{\sqrt{5}} & -\dfrac{3}{\sqrt{30}} & \dfrac{3}{\sqrt{42}} & \dfrac{2}{\sqrt{14}} & 0 \\
    -\dfrac{1}{\sqrt{5}} & \dfrac{2}{\sqrt{30}} & -\dfrac{4}{\sqrt{42}} & \dfrac{2}{\sqrt{14}} & 0 \\
    -\dfrac{1}{\sqrt{5}} & -\dfrac{3}{\sqrt{30}} & -\dfrac{3}{\sqrt{42}} & -\dfrac{2}{\sqrt{14}} & 0 \\
    -\dfrac{1}{\sqrt{5}} & \dfrac{2}{\sqrt{30}} & \dfrac{2}{\sqrt{42}} & -\dfrac{1}{\sqrt{14}} & \dfrac{1}{\sqrt{2}}
  \end{pmatrix}
}
\]

\caption{Left: exact closed form of the $3\times3$ orthonormal matrix. Right: exact closed form of the $5\times5$ orthonormal matrix.}
\label{fig:two-matrices}
\end{figure}
\begin{figure}[htbp]
  \centering
  \begin{minipage}{0.48\textwidth}
    \centering
    \begin{tikzpicture}[scale=1.4, every node/.style={font=\small}]
      \def\ox{0.5}  
      \def\oy{0.25} 

      \newcommand{\Coord}[4]{\coordinate (#1) at ({#2+\ox*#4},{#3+\oy*#4});}

      \Coord{A}{-1}{-1}{-1} 
      \Coord{B}{ 1}{-1}{-1} 
      \Coord{C}{ 1}{ 1}{-1} 
      \Coord{D}{-1}{ 1}{-1} 
      \Coord{E}{-1}{-1}{ 1} 
      \Coord{F}{ 1}{-1}{ 1} 
      \Coord{G}{ 1}{ 1}{ 1} 
      \Coord{H}{-1}{ 1}{ 1} 

      \draw[gray!80] (A) -- (B) -- (C) -- (D) -- cycle; 
      \draw[gray!80] (E) -- (F) -- (G) -- (H) -- cycle;           
      \draw[gray!80] (A) -- (E);
      \draw[gray!80] (B) -- (F);
      \draw[gray!80] (C) -- (G);
      \draw[gray!80] (D) -- (H);

      \fill[red]  (D) circle (2.6pt) node[above right=2pt] {$S_1$};
      \fill[red]  (F) circle (2.6pt) node[below left=2pt] {$-S_1$};

      \fill[blue] (A) circle (2.6pt) node[left=2pt] {$S_2$};
      \fill[blue] (E) circle (2.6pt) node[left=2pt] {$S_2$};
      \fill[blue] (C) circle (2.6pt) node[right=2pt] {$-S_2$};
      \fill[blue] (G) circle (2.6pt) node[right=2pt] {$-S_2$};

      \fill[green!60!black] (H) circle (2.6pt) node[above left=2pt] {$S_3$};
      \fill[green!60!black] (B) circle (2.6pt) node[below right=2pt] {$-S_3$};

    \end{tikzpicture}
  \end{minipage}%
  \hfill
  \begin{minipage}{0.48\textwidth}
    \centering
    \begin{tikzpicture}[scale=1.4, every node/.style={font=\small}]
      \def\ox{0.45}  
      \def\oy{0.22}  
      \def\px{-0.35} 
      \def\py{0.55}  

      \newcommand{\Coord}[5]{\coordinate (#1) at ({#2+\ox*#4+\px*#5},{#3+\oy*#4+\py*#5});}

      \Coord{v1111}{ 1}{ 1}{ 1}{ 1}
      \Coord{v111m}{ 1}{ 1}{ 1}{-1}
      \Coord{v11m1}{ 1}{ 1}{-1}{ 1}
      \Coord{v11mm}{ 1}{ 1}{-1}{-1}
      \Coord{v1m11}{ 1}{-1}{ 1}{ 1}
      \Coord{v1m1m}{ 1}{-1}{ 1}{-1}
      \Coord{v1mm1}{ 1}{-1}{-1}{ 1}
      \Coord{v1mmm}{ 1}{-1}{-1}{-1}
      \Coord{vm111}{-1}{ 1}{ 1}{ 1}
      \Coord{vm11m}{-1}{ 1}{ 1}{-1}
      \Coord{vm1m1}{-1}{ 1}{-1}{ 1}
      \Coord{vm1mm}{-1}{ 1}{-1}{-1}
      \Coord{vmm11}{-1}{-1}{ 1}{ 1}
      \Coord{vmm1m}{-1}{-1}{ 1}{-1}
      \Coord{vmmm1}{-1}{-1}{-1}{ 1}
      \Coord{vmmmm}{-1}{-1}{-1}{-1}

      \draw[gray!80] (vmmmm) -- (v1mmm) -- (v11mm) -- (vm1mm) -- cycle;
      \draw[gray!80] (vmm1m) -- (v1m1m) -- (v111m) -- (vm11m) -- cycle;
      \draw[gray!80] (vmmmm) -- (vmm1m); \draw[gray!80] (v1mmm) -- (v1m1m);
      \draw[gray!80] (v11mm) -- (v111m); \draw[gray!80] (vm1mm) -- (vm11m);

      \draw[gray!80] (vmmm1) -- (v1mm1) -- (v11m1) -- (vm1m1) -- cycle;
      \draw[gray!80] (vmm11) -- (v1m11) -- (v1111) -- (vm111) -- cycle;
      \draw[gray!80] (vmmm1) -- (vmm11); \draw[gray!80] (v1mm1) -- (v1m11);
      \draw[gray!80] (v11m1) -- (v1111); \draw[gray!80] (vm1m1) -- (vm111);

      \draw[gray!80] (vmmmm) -- (vmmm1); \draw[gray!80] (v1mmm) -- (v1mm1);
      \draw[gray!80] (v11mm) -- (v11m1); \draw[gray!80] (vm1mm) -- (vm1m1);
      \draw[gray!80] (vmm1m) -- (vmm11); \draw[gray!80] (v1m1m) -- (v1m11);
      \draw[gray!80] (v111m) -- (v1111); \draw[gray!80] (vm11m) -- (vm111);

      \fill[red]  (vm11m) circle (2.6pt) node[right=1pt] {$S_1$};
      \fill[red]  (vm111) circle (2.6pt) node[below right=1pt] {$S_1$};
      \fill[red]  (v1111) circle (2.6pt) node[below right=1pt] {$S_1$};
      \fill[red]  (v1mmm) circle (2.6pt) node[below left=1pt] {$-S_1$};

      \fill[blue] (vmm1m) circle (2.6pt) node[above left=1pt] {$S_2$};
      \fill[blue] (vm1m1) circle (2.6pt) node[below left=1pt] {$-S_2$};
      \fill[blue] (v11mm) circle (2.6pt) node[below left=1pt] {$-S_2$};
      \fill[blue] (v11m1) circle (2.6pt) node[below left=1pt] {$-S_2$};

      \fill[green!60!black] (vm1mm) circle (2.6pt) node[above left=1pt] {$S_3$};
      \fill[green!60!black] (vmm11) circle (2.6pt) node[below left=1pt] {$-S_3$};
      \fill[green!60!black] (v1m1m) circle (2.6pt) node[left=1pt] {$-S_3$};
      \fill[green!60!black] (v1m11) circle (2.6pt) node[below left=1pt] {$-S_3$};

      \fill[orange] (vmmmm) circle (2.6pt) node[above left=1pt] {$S_4$};
      \fill[orange] (vmmm1) circle (2.6pt) node[right=1pt] {$S_4$};
      \fill[orange] (v1mm1) circle (2.6pt) node[below right=1pt] {$S_4$};
      \fill[orange] (v111m) circle (2.6pt) node[below left=1pt] {$-S_4$};

    \end{tikzpicture}
  \end{minipage}

  \caption{Left: induced partitions of $\{\pm1\}^3$ for the $3\times 3$ matrix used above. Right: induced partitions of $\{\pm1\}^4$ for a $4\times4$ Orthonormal Almost–Hadamard matrix(which is just a Hadamard matrix of order 4).}
\end{figure}
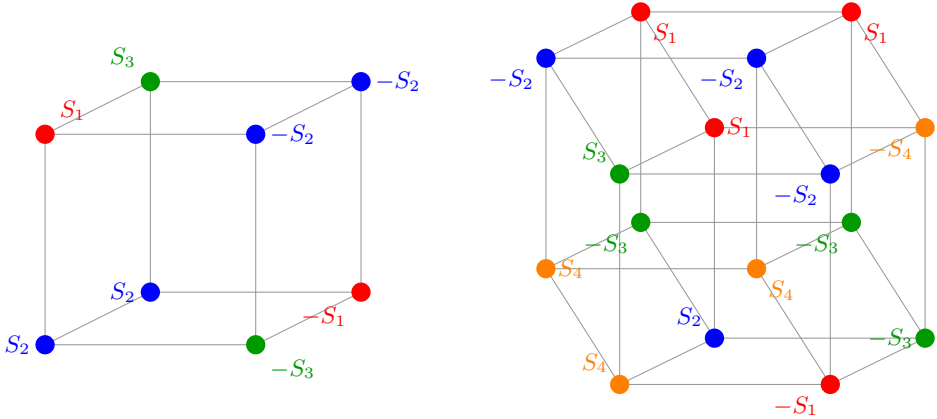

\begin{conjecture}[Hadamard's conjecture \cite{Hadamard1893}]\label{conj:hadamard}
For every positive integer $n$ divisible by $4$ there a matrix
$H\in\{\pm1\}^{n\times n}$ with $H H^{T}=n I_n$.
\end{conjecture}
Under the assumption that Hadamard's famous conjecture holds, which is a numerically valid assumption for $n$ moderately large (\(< 668 \) \cite{Kharaghani2005}), our construction for all values of \(n\) matches (and numerically exceeds) the \(\beta\)-rate of the random sign matrices. Even if the conjecture does not hold, this gives us a concrete construction with the same \(\beta\)-rate for values of \(n\) that are close to a value for which a Hadamard matrix exists by a constant independent of \(n\).

\begin{restatable}[Asymptotics for explicit constructions]{theorem}{asym}\label{thm:asymptotics}
Let $n\ge 3$.  For both the normalized random sign matrix $S$ with high probability, and for any orthonormal almost–Hadamard matrix $Q$ under the assumption of Hadamard's conjecture, one has the expansion
\begin{equation}\label{eq:asym-explicit}
\beta(\,\cdot\,) =
\sqrt{2\log(2n)} -\;\frac{\log\log(2n)}{2\sqrt{2\log(2n)}}+\mathcal{O}\left(\frac{1}{\sqrt{\log(n)}}\right).
\end{equation}
\end{restatable}

This makes explicit the \((1+o(1))\) rate in Steinerberger's original paper, where the rate \(\beta_{\mathrm{opt}} = (1+o(1))\sqrt{2\log n}\) was proved.  Combining the refined CLT control used to prove Theorem~\ref{thm:asymptotics} and the Gaussian extreme–value expansion, we obtain the rate in Theorem~\ref{thm:asymptotics} as the explicit leading order correction for the $\beta$-value of globally optimal Bad Science matrices, under the additional conditions that the matrix obeys a notion of `non-degeneracy' and the entries are bounded. These conditions are exactly the assumptions of the Central Limit Theorem in\cite{FangKoikeLiuZhao2023} and the lemma~\ref{lem:optimality} makes them explicit in our setting.

\subsection{Open Questions}
We present two directions for future results.

1. The seemingly harder direction is an exact solution to the isoperimetric optimization problem presented in Theorem~\ref{thm:Fourier Characterization}, concretely, finding the best partition of the hypercube in terms of the Level-1 weight functional and converting this partition to a Bad Science matrix (the rows being the normalized centroids of the partition sets). We discuss why this problem seems hard and what is known about it in section~\ref{sec:geometry}.

\begin{figure}[H]
  \centering
  \includegraphics[width=0.9\textwidth]{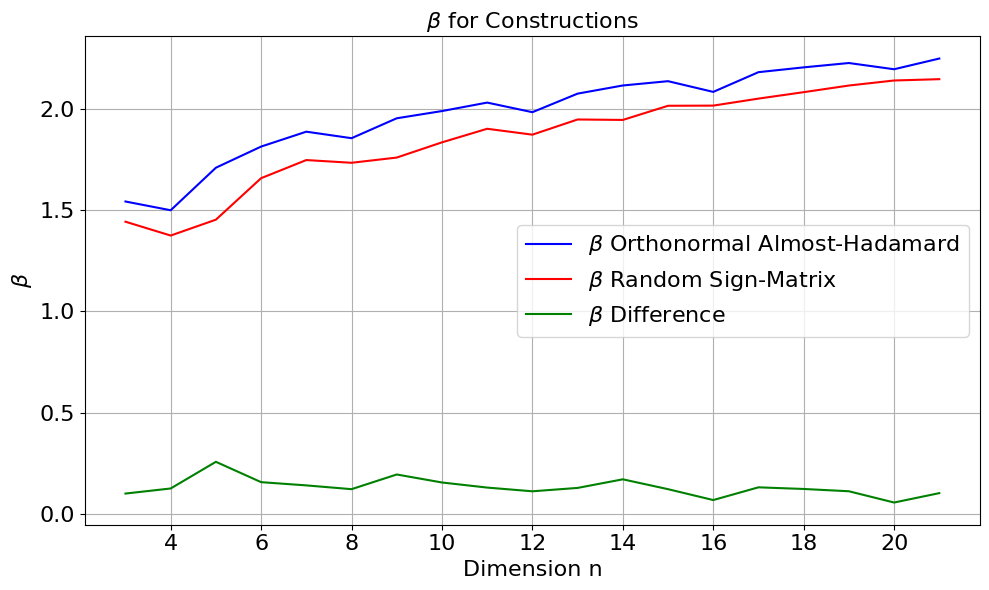}
  \caption{Running exact simulations of the expectation, which is computationally infeasible for $n$\(>\)20, shows a strict gap between the deterministic and random constructions.}
  \label{fig:constructions}
\end{figure}

\begin{figure}[H]
  \centering
  \includegraphics[width=0.9\textwidth]{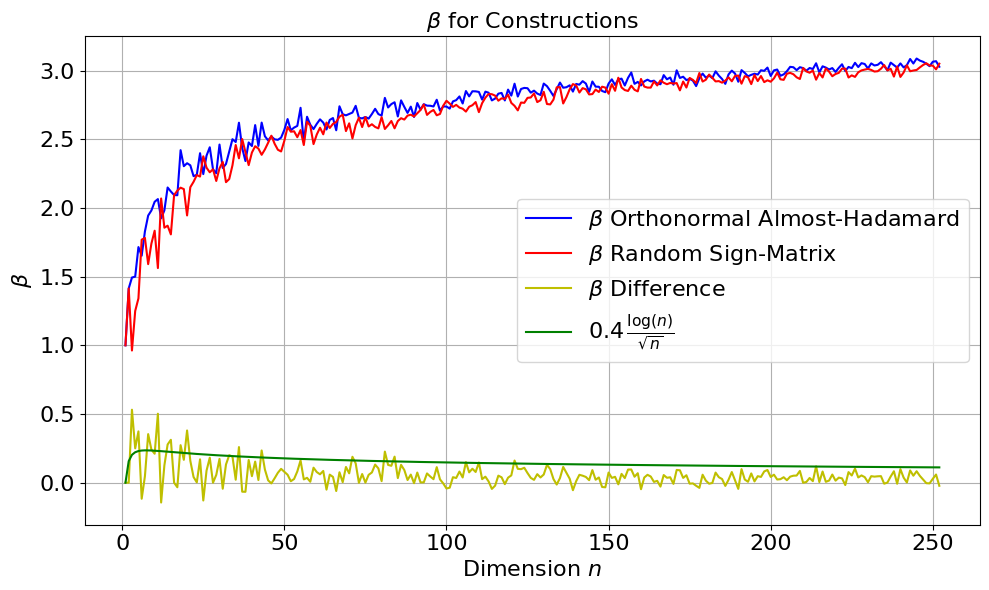}
  \caption{Running Monte-Carlo simulations of the expectation, for larger values of $n$, shows a gap visually matching the asymptotic bound in the proof of Theorem ~\ref{thm:asymptotics}.}
\end{figure}

2. The seemingly easier direction is relaxing the heuristic conditions under which the optimality of the first order correction for \(\beta\) is obtained in Theorem~\ref{thm:asymptotics}. To invoke the high dimensional central limit theorem we use, we need entries bounded in $n$ and some non-degeneracy conditions on the covariance matrix used in Theorem~\ref{thm:asymptotics}. It seems reasonable to believe that bounded entries lead to the best behavior for the maximum but we do not know how to analyze the unbounded case while simultaneously controlling the covariance matrix, although we suspect this can be done with a lot of painstaking case-distinction.

\subsection{Related Results and Hardness of the Exact Problem}\label{sec:geometry}
The local description in Theorem~\ref{thm:Fourier Characterization} accords with classical single-set extremizer theory: Harper’s isoperimetric theorem and Fourier-analytic results identify Hamming balls and half-spaces as the canonical Level-1 (and boundary) extremizers \cite{ODonnell14,harper1964isoperimetric}.  On the other hand, the standard multi-cell constructions in coding theory (notably 1-perfect / Hamming code partitions) do produce equal-size Voronoi regions but with exponentially many cells and fixed small radius, not with only $n$ seeds of mass $\Theta(1/n)$ \cite{macwilliams1977theory,krotov2012partition,bespalov2019perfect}.  Strong stability theorems for isoperimetry on the cube corroborate the lemma’s local rigidity (near-extremal sets must be close to Hamming balls or unions of subcubes) \cite{keevash2017stability,przykucki2020vertex}, yet these are single-set results and do not by themselves yield a global Voronoi partition with the combinatorial parameters considered here.  Finally, from an algorithmic perspective the centroidal/Voronoi nature of the requirement places the problem in the realm of difficult clustering/centroidal V.T.\ objectives for which NP-hardness results are known in related continuous settings \cite{aloise2009kmeans,du1999cvt}; while this does not constitute a proof of impossibility in the discrete cube, it provides a strong complexity heuristic.  In summary, although the lemma enforces that almost every cell must be locally indistinguishable from a Level-1 extremizer, we are unaware of any prior work that constructs a ``globally optimal" Voronoi partition of $\{\pm1\}^n$ with only $n$ seeds and cell masses $\Theta(1/n)$ for which almost all cells simultaneously saturate the Level-1 Inequality; the construction of such partitions thus appears to be an interesting open problem.

\section{Proofs} 
\subsection{Structural Results}
\begin{proposition}\label{prop:no-ties}
We note here that for the rest of the paper, we can narrow our optimization argument to the following dense open subset of  
\[
\mathcal A = \bigl\{\,A\in\mathbb{R}^{n\times n} : \|a_i\|_2 = 1,\ i=1,\dots,n\bigr\},
\]  
namely  
\[
\mathcal A^\circ
=\mathcal A
\;\setminus\;
\bigcup_{\substack{x\in\{-1,1\}^n\\i\neq j}}
\bigl\{\,A : \langle a_i,x\rangle^2 - \langle a_j,x\rangle^2 = 0\bigr\}.
\]  
On \(\mathcal A^\circ\), for each hypercube vector \(x\in\{-1,1\}^n\) the value  
\(\|Ax\|_\infty = \max_i|\langle a_i,x\rangle|\)  
is attained at a unique row index.

\end{proposition}
\begin{proof}
Fix any \(x\in\{-1,1\}^n\) and two distinct rows \(i\neq j\).  The “tie‐condition”
\[
|\langle a_i,x\rangle| \;=\; |\langle a_j,x\rangle|
\quad\Longleftrightarrow\quad
\langle a_i,x\rangle^2 \;-\;\langle a_j,x\rangle^2 = 0
\]
is a nontrivial polynomial equation in the \(2n\) entries of \(a_i\) and \(a_j\).  Hence its zero‐set
\[
H_{x,i,j}
=\bigl\{A\in\mathcal A : \langle a_i,x\rangle^2 - \langle a_j,x\rangle^2 = 0\bigr\}
\]
is a real hypersurface (codimension \(\ge1\)) in the compact manifold \(\mathcal A\).  There are only finitely many such triples \((x,i,j)\), so  
\(\mathcal Z=\bigcup H_{x,i,j}\)  
is a finite union of measure‐zero hypersurfaces, and  
\(\mathcal A^\circ=\mathcal A\setminus\mathcal Z\)  
is open and dense.
Meanwhile, the objective
\[
\beta(A)
=\sum_{x\in\{-1,1\}^n}\|Ax\|_\infty
\]
is continuous (indeed Lipschitz) on \(\mathcal A\).  By compactness, it attains its maximum, and continuity implies any maximizer can be approximated arbitrarily well by points in \(\mathcal A^\circ\).  Thus, there is no loss of generality in restricting our search to \(\mathcal A^\circ\), where for every \(x\) the maximizer of \(\|Ax\|_\infty\) is unique.  
\end{proof}

For the following results, we first outline the Fourier analysis of the indicator functions for subsets of the cube, with reference from O'Donnell's book\cite{ODonnell14}.
\[\mathbf{1}_{B}(x) = 
\begin{cases}
1 & \text{if } x\in B\\[1mm]
0 & \text{otherwise}
\end{cases}.
\]

The function \(\mathbf{1}_{B}\) has the Fourier expansion
\[
\mathbf{1}_{B}(x) = \sum_{S\subset [n]} \widehat{\mathbf{1}_{B}}(S)\chi_S(x),
\]
where \(\chi_S(x)=\prod_{i\in S} x_i\) and
\begin{align*}
    \widehat{\mathbf{1}_{B}}(S) = \frac{1}{2^n}\sum_{x\in \{-1,1\}^n} \mathbf{1}_{B}(x)\chi_S(x)
\end{align*}
Note that the sum of the inner products over pairs in \(B\) can be written as
\begin{align*}
    \sum_{x,y\in B} \langle x,y \rangle &= \sum_{x,y\in \{-1,1\}^n} \mathbf{1}_{B}(x)\mathbf{1}_{B}(y) \langle x,y \rangle \\
&=  \sum_{i=1}^n \Biggl(\sum_{x\in \{-1,1\}^n} x_i \mathbf{1}_{B}(x)\Biggr)^2.
\end{align*}
Observing that by the definition of the Fourier coefficients, we have
\[
\widehat{\mathbf{1}_{B}}(\{i\}) = \frac{1}{2^n}\sum_{x\in \{-1,1\}^n} x_i \mathbf{1}_{B}(x),
\]
it follows that
\[
\sum_{x,y\in B} \langle x,y \rangle = 2^{2n}\sum_{i=1}^n \widehat{\mathbf{1}_{B}}(\{i\})^2.
\]
If we define the \emph{level-1 Fourier weight} by
\[
W_1[\mathbf{1}_{B}] := \sum_{i=1}^n \widehat{\mathbf{1}_{B}}(\{i\})^2 \qquad \mbox{then} \qquad
\sum_{x,y\in B} \langle x,y \rangle = 2^{2n} \cdot W_1[\mathbf{1}_{B}].
\]
We make extensive use of a well-known``Level-1 Inequality," perhaps first introduced in a paper by Talagrand \cite{Talagrand}.

\begin{lemma}[Level-1 Inequality]\label{lem:Level-1 Inequality}
Let $f : \{-1,1\}^n \to \{0,1\}$ be a Boolean function with mean $\mathbb{E}[f] = \alpha \leq \frac{1}{2}$. Then
\[
W_1[f] \leq 2\alpha^2 \log\left(\frac{1}{\alpha}\right),
\]
which is asymptotically sharp for Hamming Balls.  
\end{lemma}
We use this to give a structural result about the extremal Bad Science matrices.

\FC*

\begin{proof}
Recall that
\[
\beta(A) :=  \frac{1}{2^n} \sum_{x \in \{-1,1\}^n} \|Ax\|_\infty.
\]
For each row \(A_i\) of \(A\), there exists the largest subset \(S_i \subset \{-1,1\}^n\)  such that
\[
\forall x \in S_i, \quad \|A x\|_\infty = \langle A_i, x \rangle.
\]
By our note from earlier, we also know that these sets are disjoint, and cover half the hypercube, while the other half is covered by their negatives, which are also disjoint.  We may interpret each \(S_i\) as a Voronoi cell when we consider a Voronoi tessellation of the hypercube seeded at the rows \(A_i\) and their negatives.
Enumerating elements of \(S_i\) as \(v_1,...,v_{|S_i|}\), we can write
\[\beta(A) = 
\frac{1}{2^n}\sum_{i=1}^{n} \sum_{j=1}^{|S_i|} \langle A_i, 2v_j \rangle = \frac{1}{2^n}\sum_{i=1}^{n}  \left\langle A_i,2\sum_{j=1}^{|S_i|} v_j \right\rangle.
\]
We define the sum vector
\[
B_i := 2 \sum_{j=1}^{|S_i|} v_j.
\]
This simplifies our equation to the sum of inner products between the rows \(A_i\) and the unscaled centroid of their associated Voronoi cell \(S_i\)
\[
\beta(A) = \frac{1}{2^n}\sum_{i=1}^{n}  \left\langle A_i,B_i \right\rangle.
\] 
Using the Cauchy-Schwarz inequality,
\[
\beta(A) = \frac{1}{2^n}\sum_{i=1}^{n}  \left\langle A_i,B_i \right\rangle \leq  \frac{1}{2^n}\sum_{i=1}^{n} \|A_i\| \| B_i\| =  \frac{1}{2^n}\sum_{i=1}^{n}   \| B_i\|.
\] 

So the best case would be if every row could be the normalized centroid of its induced Voronoi cell. Recalling from the definition of \(B_i\)
\[
\|B_i\|^2 = 4\sum_{x,y\in S_i} \langle x,y \rangle.
\]
It follows from the Fourier analysis introduction that
\[
\beta(A) \le \frac{1}{2^n} \sum_{i=1}^{n} \|B_i\| 
= 2\sum_{i=1}^{n} \sqrt{W_1[\mathbf{1}_{S_i}]}.
\]
\end{proof}
Combining this lemma with theorem \ref{thm:asymptotics} (which we prove in section \ref{subsec:Heuristic Constructions}) we get a result that elucidates the connection of the Bad Science matrix problem with an isoperimetric tiling problem on the Boolean hypercube.

\renewcommand{\StructureExtra}{%
and for \(\varepsilon_n \to 0\) such that \(\dfrac{\log\log(2n)}{\,\log(2n)} = o(\varepsilon_n)\), we have for all but $o(n)$ indices $i$,
\[
\alpha_i \sqrt{2\log\frac{1}{\alpha_i}}-\varepsilon_n\,\frac{1}{2n}\sqrt{2\log(2n)}
\le \sqrt{W_1[1_{S_i}]}\le \alpha_i \sqrt{2\log\frac{1}{\alpha_i}}.
\]
}%
\Structure*
\renewcommand{\StructureExtra}{}
\begin{proof} 
Applying the Level--1 inequality, we have
\begin{equation*}\label{eq:1}
W_1[1_{S_i}]\le 2\alpha_i^2 \log\frac{1}{\alpha_i}, \text{ for }\alpha_i=\frac{|S_i|}{2^{n-1}}.
\end{equation*}
Define
\begin{equation*}\label{eq:2}
f(x)=x\sqrt{2\log\frac1{x}},\qquad x\in(0,1],
\end{equation*}
and for \(\boldsymbol\alpha=(\alpha_1,\dots,\alpha_n)\) set
\begin{equation*}\label{eq:3}
F(\boldsymbol\alpha)=\sum_{i=1}^n f(\alpha_i).
\end{equation*}
Since \(f\) is strictly concave on \((0,1]\), Jensen's inequality yields
\begin{equation*}\label{eq:4}
F(\boldsymbol\alpha)\le n\,f\!\Bigl(\frac{1}{2n}\Bigr)
= \tfrac12\sqrt{2\log(2n)},
\end{equation*}
with equality iff \(\alpha_i=\tfrac{1}{2n}\) for all \(i\). Combining this with Theorem~\ref{thm:asymptotics} (the asymptotic lower bound for \(\beta(A)\)) shows that we may put
\begin{equation}\label{eq:5}
\eta := 1-\frac{F(\boldsymbol\alpha)}{n f(\tfrac{1}{2n})}
= \mathcal{O}\!\Bigl(\frac{\log\log(2n)}{\log(2n)}\Bigr).
\end{equation}

Write
\begin{equation*}\label{eq:6}
\alpha_i=\tfrac{1}{2n}+\delta_i,
\qquad \sum_i\delta_i=0.
\end{equation*}
Because \(\alpha_i\ge0\) we have the uniform lower bound
\begin{equation}\label{eq:7}
\delta_i\ge -\tfrac{1}{2n}\qquad\text{for every }i.
\end{equation}
Fix the threshold \(n^{-3/4}\) and split the index set into
\begin{equation*}\label{eq:8}
G:=\{i:|\delta_i|\le n^{-3/4}\},\qquad B:=\{i:|\delta_i|> n^{-3/4}\}.
\end{equation*}
The contribution of the good set \(G\) may be upper-bounded as
\begin{equation*}\label{eq:9}
\sum_{i\in G}\delta_i^2 \le |G|\cdot n^{-3/2}\le n\cdot n^{-3/2}=n^{-1/2},
\end{equation*}
which is negligible compared with the desired bound. Consequently it suffices to bound \(\sum_{i\in B}\delta_i^2\).

For every \(i\in B\) we have \(|\delta_i|>n^{-3/4}>\tfrac{1}{2n}\), hence by (\ref{eq:7}) \(\delta_i>n^{-3/4}\) and therefore \(\alpha_i=\tfrac{1}{2n}+\delta_i>\tfrac{1}{2n}\). Thus, for fixed $n$ the closed interval between \(\alpha_i\) and \(\tfrac{1}{2n}\) is bounded away from the singularity at 0, so the second--order Taylor expansion in Lagrange form about \(x=\tfrac{1}{2n}\) is valid for every \(i\in B\): there exists \(\xi_i\) between \(\alpha_i\) and \(\tfrac{1}{2n}\) with
\begin{equation*}\label{eq:10}
f(\alpha_i)=f\!\Bigl(\frac{1}{2n}\Bigr)+f'\!\Bigl(\frac{1}{2n}\Bigr)\delta_i
+\tfrac12 f''(\xi_i)\delta_i^2.
\end{equation*}
From the explicit formula
\begin{equation}\label{eq:11}
f''(x) = -\,\frac{\sqrt{2}}{4}\,\frac{2(-\log x)+1}{x(-\log x)^{3/2}},\qquad x\in(0,1],
\end{equation}
we have \(f''(x)<0\) on \((0,1]\) and the positive quantity \(|f''(x)|\) attains its global minimum on \((0,\tfrac12]\) at an interior point (numerically near \(x\approx0.42062\)). Denote this absolute constant by
\begin{equation*}\label{eq:12}
c_0:=\min_{x\in(0,1/2]} |f''(x)|>0.
\end{equation*}
Summing the Taylor identities (\ref{eq:11}) and rearranging gives
\begin{equation*}\label{eq:13}
\sum_{i}\Bigl(f\!\Bigl(\frac{1}{2n}\Bigr)-f(\alpha_i)\Bigr)
= \tfrac12\sum_{i} |f''(\xi_i)|\,\delta_i^2
\ge \tfrac12 c_0\sum_{i\in B}\delta_i^2.
\end{equation*}
Writing the sum term as the total deficit, and using the definition of $\eta$ from (\ref{eq:5}) yields
\begin{equation*}\label{eq:15}
\frac{1}{2}c_0\sum_{i\in B}\delta_i^2
\le n f\!\Bigl(\frac{1}{2n}\Bigr)-F(\boldsymbol\alpha)
= \eta\, n f\!\Bigl(\frac{1}{2n}\Bigr).
\end{equation*}
Hence, we may bound the contribution over the bad set by
\begin{equation*}\label{eq:16}
\sum_{i\in B}\delta_i^2
\le \frac{2\eta\, n f(\tfrac{1}{2n})}{c_0}.
\end{equation*}
Combining this with the bound over the good set gives
\begin{equation*}\label{eq:17}
\sum_{i=1}^n\Bigl(\alpha_i-\frac{1}{2n}\Bigr)^2
= \sum_{i\in G}\delta_i^2 + \sum_{i\in B}\delta_i^2
\le n^{-1/2} + \frac{2\eta\, n f\!\bigl(\tfrac{1}{2n}\bigr)}{c_0}.
\end{equation*}
Using
\begin{equation*}\label{eq:18}
n f\!\Bigl(\frac{1}{2n}\Bigr)=\tfrac12\sqrt{2\log(2n)}
\end{equation*}
and the bound \eqref{eq:5} for \(\eta\) yields
\begin{equation*}\label{eq:19}
\sum_{i=1}^n\Bigl(\alpha_i-\frac{1}{2n}\Bigr)^2
= \mathcal{O}\!\Biggl(\frac{\log\log(2n)}{\sqrt{\log(2n)}}\Biggr),
\end{equation*}
since the \(n^{-1/2}\) term is negligible compared with the displayed asymptotic rate.
It remains to deduce the asymptotic formula for the level--one weights for the vast majority of indices \(i\). Write
\begin{equation*}\label{eq:20}
U_i:=f(\alpha_i)-\sqrt{W_1[1_{S_i}]}\ge 0.
\end{equation*}
Then
\begin{equation*}\label{eq:21}
n f\!\Bigl(\frac{1}{2n}\Bigr)-\sum_{i=1}^n\sqrt{W_1[1_{S_i}]}
=\bigl(n f\!\bigl(\tfrac{1}{2n}\bigr)-F(\boldsymbol\alpha)\bigr)+\sum_{i=1}^n U_i.
\end{equation*}
Using \eqref{eq:5} we obtain
\begin{equation*}\label{eq:22}
n f\!\Bigl(\frac{1}{2n}\Bigr)-\sum_{i=1}^n\sqrt{W_1[1_{S_i}]}
=\mathcal{O}\Bigl(\frac{\log\log(2n)}{\log(2n)}\Bigr)\cdot n f\!\Bigl(\frac{1}{2n}\Bigr),
\end{equation*}
and hence
\begin{equation*}\label{eq:23}
\sum_{i=1}^n U_i
=\mathcal{O}\Bigl(\frac{\log\log(2n)}{\log(2n)}\Bigr)\cdot n f\!\Bigl(\frac{1}{2n}\Bigr).
\end{equation*}
For \(\varepsilon>0\) let \(B_0:=\{i:U_i\ge\varepsilon f(1/(2n))\}\). Then
\begin{equation*}\label{eq:24}
|B_0|\cdot\varepsilon f\!\Bigl(\frac{1}{2n}\Bigr)\le\sum_{i\in B_0} U_i\le\sum_{i=1}^n U_i
=\mathcal{O}\Bigl(\frac{\log\log(2n)}{\log(2n)}\Bigr)\cdot n f\!\Bigl(\frac{1}{2n}\Bigr),
\end{equation*}
so cancelling \(f(1/(2n))\) yields the exceptional--set bound
\begin{equation*}\label{eq:25}
|B_0| \le C\,\frac{\log\log(2n)}{\varepsilon\,\log(2n)}\; n.
\end{equation*}
Choosing \(\varepsilon=\varepsilon_n\) as in the lemma statement gives \(|B_0|=o(n)\), and therefore for all but \(o(n)\) indices \(i\) we have
\begin{equation*}\label{eq:26}
f(\alpha_i)-\varepsilon_n\,f\!\Bigl(\frac{1}{2n}\Bigr)\le\sqrt{W_1[1_{S_i}]}\le f(\alpha_i).
\end{equation*}
This completes the proof.
\end{proof}
The tight Cauchy-Schwarz inequality shows that this approaches a constrained centroidal Voronoi tessellation, meaning the normalized centroids of each cell coincide with the row vectors. The tight Jensen and Level-1 inequalities add that each cell is roughly equal volume and that most cells are isoperimetrically optimal in the sense that they asymptotically maximize Level-1 weight for sets of their fixed size. 
\subsection{Central Limit Theorem Framework}
We will now develop a probabilistic viewpoint that shows the connection of the Bad Science Problem with well-studied Gaussian extreme value theory. The setting is as follows:
For a real \(n\times n\) matrix \(A\) with each row having \(\ell_2\)-norm \(\sqrt n\), we are interested in 
\[
\beta(\hat{A}), 
\quad\text{where}\quad 
\hat{A}= A/\sqrt{n}.
\]
Let \(A^T_i\) denote the \(i\)-th row of \(A^T\).  Let \(\{\eps_i\}_{i=1}^n\) denote a collection of i.i.d.\ Rademacher random variables 
\(\bigl(\mathbb{P}(\eps_i = 1) = \mathbb{P}(\eps_i = -1) = \tfrac12\bigr)\).  
Then 
\[
X_i = A^T_i\,\eps_i,\quad i=1,\dots,n,
\]
are independent mean-zero vectors in \(\mathbb{R}^n\).
If we consider
\[
S_n = \frac{1}{\sqrt{n}}\sum_{i=1}^n X_i,
\]
the \(j\)th coordinate of \(S_n\) is precisely the inner product of the row \(\hat A_j\) with a uniformly random vector in \(\{-1,1\}^n\).  We are interested in computing
\[
\beta(\hat{A}) \;=\; \mathbb{E}\bigl[\|S_n\|_{\infty}\bigr].
\]
The reason we chose row-norm \(\sqrt n\) is so that normalizing by \(\sqrt n\) gives (under light conditions on $A$ that we clarify later), uniformly for all $t$ as $n$ goes to infinity
\[
\Bigl|\mathbb{P}(\|S_n\|_{\infty}>t) \;-\; \mathbb{P}(\|Z_n\|_{\infty}>t)\Bigr|\;\longrightarrow\;0,
\]
where \(Z_n\sim\mathcal{N}(0,\mathbb{E}[S_nS_n^T])\), as studied in high-dimensional central-limit theorem type results like \cite{FangKoikeLiuZhao2023}.
Steinerberger already showed that $n\times n$ matrices with i.i.d. Rademacher entries attain the optimal bound for beta for large $n$. We investigate this random construction further, provide the precise asymptotics for the \(o(1)\) term in the optimal beta rate, and present a deterministic construction that converges to the optimal rate as well. We present numerical evidence that this method converges faster, as well as a heuristic explanation for this speed-up.
The essential ingredient is a version of a `high-dimensional central-limit theorem' result from Fang, Koike, Liu, and Zhao \cite{FangKoikeLiuZhao2023}, rephrased for our setting.
\begin{theorem}[High‐Dimensional CLT in the Degenerate Case \cite{FangKoikeLiuZhao2023}]\label{thm:degenerate-CLT}
Let \(n> 3\) be an integer.  Let \(X_1,\dots,X_n\) be independent mean-zero random vectors in \(\mathbb{R}^n\).  Define
\[
W \;=\; \frac{1}{\sqrt{n}}\sum_{i=1}^n X_i,
\]
and let \(\Sigma\) denote its covariance matrix,  satisfying, for all distinct \(1 \le j,k,\ell \le n\),
\[
\Sigma_{jj} = 1,\qquad
\det(\Sigma_{j,k}) > \alpha^2 > 0,\qquad
\frac{\det(\Sigma_{j,k,\ell})}{\det(\Sigma_{j,k})} > \beta^2 > 0,
\]
where \(\Sigma_{j,k}\) is the \(2\times2\) sub-matrix of \(\Sigma\) by intersecting the \(j\)th and \(k\)th rows with the \(j\)th and \(k\)th columns and \(\Sigma_{j,k,\ell}\) the \(3\times3\) sub-matrix defined analogously in the obvious way.  Let \(G\sim N(0,\Sigma)\). Suppose there exists a constant \(B>0\) such that
\[
\mathbb{E}\Bigl[\exp\bigl(|X_{ij}|/B\bigr)\Bigr] \;\le\; 2
\quad\text{for all }1\le i,j\le n,
\]
then there is an absolute constant \(C\) such that
\[\sup_{a,b\in\mathbb{R}^n}\bigl|P(a\le W\le b)\;-\;P(a\le G\le b)\bigr| \le
\frac{C\,B^3}{\alpha^2\beta^2\,\sqrt{n}}\;\log n^{6.5}.
\]
\end{theorem}
We note here that up to \(\polylog(n)\) factors, this is the best known bound for this quantity. We choose this version as it is convenient to use and, as we shall see, \(\polylog(n)\) factor improvements will not change the expected leading order of the infinity norm. The main idea behind considering derandomized constructions is that the random matrices from Steinerberger's construction have, upon being fixed, non-zero pairwise inner products with high probability. This means that when applying this central limit theorem, we will be converging to a centered gaussian (mean zero), with non-zero covariances. A classical result, which we now state, shows that identically zero covariances have the greatest infinity-norm, thus one should consider orthogonal matrices, the canonical example of which is a Hadamard matrix.

\begin{theorem}[Gaussian Correlation Inequality \cite{Pitt1982}]\label{thm:Gaussian Correlation Inequality}
Let $Z\sim\mathcal{N}(0,\Sigma)$ be an $n$-dimensional centered Gaussian, and let $A,B\subset\mathbb{R}^n$ be symmetric convex sets.  Then
\[
\Pr[\,Z\in A\cap B\,]\;\ge\;\Pr[\,Z\in A\,]\;\Pr[\,Z\in B\,]\,. 
\]
\end{theorem}

\begin{corollary}\label{cor:Infinity-norm}
Let \(Z=(Z_1,\dots,Z_n)\) be an \(n\)-dimensional centered Gaussian vector, i.e. 
\[
Z\sim\mathcal N(0,\Sigma),
\]
where \(\Sigma\) is the covariance matrix and \(\Sigma_{ii}=\Var(Z_i)=1\) for each \(i=1,\dots,n\).  Then
\[
\E\bigl[\|Z\|_\infty\bigr]=\E\bigl[\max_{1\le i\le n}|Z_i|\bigr]
\]
is maximized among all such covariance matrices \(\Sigma\) when \(\Sigma=I_n\) (equivalently, when \(Z_1,\dots,Z_n\) are independent standard normal random variables).
\end{corollary}
\begin{proof}
Fix $t>0$ and for $i=1,\dots,n$ set
\[
A_i=\{x\in\mathbb R^n:\,-t\le x_i\le t\},
\]
which are symmetric convex sets. By the Gaussian Correlation Inequality (applied repeatedly),
\[
\Pr_\Sigma\Big(\bigcap_{i=1}^n A_i\Big)
\;\ge\;
\prod_{i=1}^n\Pr_\Sigma(A_i)
\;=\;
\prod_{i=1}^n\Pr_\Sigma\{|Z_i|\le t\}
\;=\;
\bigl(\Pr\{|g|\le t\}\bigr)^n,
\]
where $g\sim\mathcal N(0,1)$ and the last equality uses that each marginal $Z_i$ is $\mathcal N(0,1)$. 
If $\Sigma=I_n$ then the coordinates $Z_1,\dots,Z_n$ are independent, so the left-hand side factors and we have equality in the displayed inequality in that case.
Therefore, for every $t>0$,
\[
\Pr_{\Sigma}\{\|Z\|_\infty>t\}
=1-\Pr_{\Sigma}\{\|Z\|_\infty\le t\}
\le
1-\bigl(\Pr\{|g|\le t\}\bigr)^n
=\Pr_{\mathrm{Id}}\{\|Z\|_\infty>t\}.
\]
Integrating over $t\in[0,\infty)$ and using the identity 
$$\displaystyle \E[X]=\int_0^\infty \Pr\{X>t\}\,dt$$ for nonnegative random variables $X$ gives
\[
\E_\Sigma\bigl[\|Z\|_\infty\bigr]
\le
\E_{\mathrm{Id}}\bigl[\|Z\|_\infty\bigr],
\]
with equality when $\Sigma=I_n$. This proves the corollary.
\end{proof}

This maximum is well understood asymptotically from Extreme Value Theory. We get the value here by adapting the arguments from \emph{et al.}\ \cite[Ch.~1]{Leadbetter1983}, and Hall \cite{Hall1979}.
\begin{lemma}\label{lem:max-gauss}
Let $X_{1},\dots,X_{n}$ be independent $\mathcal{N}(0,1)$ random variables, and set
$
M_{n} = \max_{1\le i\le n}|X_{i}|.
$
Then as $n\to\infty$,
\[
\mathbb{E}[M_{n}]
=
\sqrt{2\log(2n)}
\;-\;
\frac{\log\log(2n)+\log(4\pi)}{2\,\sqrt{2\log(2n)}}
\;+\;
\frac{\gamma}{\sqrt{2\log(2n)}}
\;+\;
o\bigl((\log n)^{-1/2}\bigr),
\]
where $\gamma$ is the Euler–Mascheroni constant.
\end{lemma}

Combining this with the comparatively negligible error in the Central Limit  Theorem (given that the matrix obeys the non-degeneracy conditions of \cite{FangKoikeLiuZhao2023}  and each entry is bounded, both of which are typical conditions), we see that the second-order approximation for the Gaussian absolute maximum is the correct \(o(1)\) term in the rate \(\beta_{\mathrm{opt}} = (1+o(1))\sqrt{2\log n}\) which was proved  in \cite{Steinerberger}.

We also state a result of Chatterjee, that gives us a precise idea on how much the covariance matrix matters for this rate.
\begin{theorem}[Chatterjee \cite{chatterjee2005error}]\label{thm:chatterjee-sudakov-fernique}
Let $(X_1,\dots,X_n)$ and $(Y_1,\dots,Y_n)$ be Gaussian random vectors with $\mathbb{E}(X_i)=\mathbb{E}(Y_i)$ for each $i$.  For $1\le i,j\le n$, let
\[
\gamma^X_{ij}=\mathbb{E}(X_i-X_j)^2,\qquad
\gamma^Y_{ij}=\mathbb{E}(Y_i-Y_j)^2,
\]
and let
\[
\gamma=\max_{1\le i,j\le n}\bigl|\gamma^X_{ij}-\gamma^Y_{ij}\bigr|.
\]
Then
\[
\left|\mathbb{E}\!\bigg(\max_{1\le i\le n} X_i\bigg)-\mathbb{E}\!\bigg(\max_{1\le i\le n} Y_i\bigg)\right|
\le \sqrt{\gamma\log n}.
\]
Moreover, if $\gamma^X_{ij}\le \gamma^Y_{ij}$ for all $i,j$, then
\[
\mathbb{E}\!\bigg(\max_{1\le i\le n} X_i\bigg)\le \mathbb{E}\!\bigg(\max_{1\le i\le n} Y_i\bigg).
\]
\end{theorem}
This is a quantitative refinement of the celebrated Sudakov-Fernique inequality due to Sudakov \cite{Sudakov1971} and Fernique \cite{Fernique1975}. In our setting, we already have that the standard normal has the greatest expected maximum, so this tells us how much we can lose depends on the maximum magnitude element of the covariance matrix.

\subsection{Comparison of Heuristic Constructions}\label{subsec:Heuristic Constructions}
Now that we have the probabilistic framework, we state our construction and provide a numerical comparison with the random matrices.
\begin{definition}[Normalized random sign matrix]
Let \(n\in\mathbb N\). An \(n\times n\) \emph{normalized random sign matrix} \(S\) has entries
\[
S_{ij}=\frac{\xi_{ij}}{\sqrt{n}},\qquad 1\le i,j\le n,
\]
where the \(\{\xi_{ij}\}\) are i.i.d.\ Rademacher random variables (\(\mathbb P(\xi_{ij}=1)=\mathbb P(\xi_{ij}=-1)=1/2\)).
\end{definition}

\begin{lemma}\label{lem:rand‐sign‐maxip}
Let \(S\in\R^{n\times n}\) be a normalized random sign matrix.  Then
\[
\Pr\Bigl\{\max_{i\neq j}\bigl|\langle S_{i,\ast},S_{j,\ast}\rangle\bigr|
>\sqrt{\frac{10\log n}{n}}\Bigr\}
\le n^{-3}.
\]
\end{lemma}

\begin{proof}
For any fixed \(i<j\), write
\[
\langle S_{i,\ast},S_{j,\ast}\rangle
=\frac1n\sum_{k=1}^n X_k,
\]
where \(X_k=s_{ik}s_{jk}\) are independent Rademacher variables.  By Hoeffding’s inequality,
\[
\Pr\bigl\{|\tfrac1n\sum_{k=1}^n X_k|>t\bigr\}
\le 2\exp\!\bigl(-\tfrac{n t^2}{2}\bigr).
\]
Taking the union bound over the \(\binom n2<n^2/2\) pairs \((i,j)\) gives
\[
\Pr\Bigl\{\max_{i<j}|\langle S_{i,\ast},S_{j,\ast}\rangle|>t\Bigr\}
\le n^2\exp\!\bigl(-\tfrac{n t^2}{2}\bigr).
\]
Choosing \(t=\sqrt{\tfrac{10\log n}{n}}\) yields
\[
n^2\exp \bigl(-\tfrac{n\,(10\log n)/n}{2}\bigr)
= n^2\,n^{-5}
= n^{-3},
\]
and the result follows.
\end{proof}

We define Orthonormal Almost-Hadamard matrices as definition~\ref{def:Orth} in section \ref{sec:constructions}.
The reason we need orthonormality is so that the high-dimensional central-limit theorem has us converge to the standard normal. The reason we chose a Hadamard matrix to start is that to optimize the rate in the high-dimensional central-limit theorem, we need strong moment bounds on individual entries of the matrix, which are best facilitated by uniform matrix entries.
\begin{lemma}[Flatness of Truncated Hadamard under Hadamard’s Conjecture]\label{lem:trunc-hadamard-flat}
Assume Hadamard’s conjecture \cite{Hadamard1893} holds, i.e.\ that for every positive multiple of 4 there exists a Hadamard matrix of that order.  Fix \(n\in\mathbb{N}\).  Let \(Q\) be an \(n\times n\) Orthonormal Almost–Hadamard matrix then
\[
\lvert Q_{ij}\rvert=\mathcal{O}\bigl(n^{-1/2}\bigr),
\qquad
1\le i,j\le n.
\]
\end{lemma}

\begin{proof}
Let \(m\), \(H\), \(U\), $Q$, and $R$ be as defined in definition (\ref{def:Orth}).

Since \(HH^{\top}=mI_m\), for \(i\neq j\) we have
\begin{equation}\label{eq:inner-product-sum}
\langle u_i,u_j\rangle
=
\frac{1}{m}\sum_{r=1}^n H_{ri}H_{rj}
=
-\frac{1}{m}\sum_{r=n+1}^m H_{ri}H_{rj},
\end{equation}
and \(m-n<4\) implies
\begin{equation}\label{eq:inner-product-bound}
\langle u_i,u_j\rangle=\mathcal{O}(n^{-1}).
\end{equation}
(Equation \eqref{eq:inner-product-bound} follows from \eqref{eq:inner-product-sum} together with the fact that the latter sum contains at most \(m-n<4\) terms and each term is bounded.)  Moreover
\begin{equation}\label{eq:ui-norm}
\|u_i\|^2=\frac{n}{m}.
\end{equation}
Hence, the Gram matrix satisfies
\begin{equation}\label{eq:gram-decomposition}
U^{\top}U \;=\; \frac{n}{m}I_n + E,
\qquad
\|E\|_{\max}=\mathcal{O}(n^{-1}),
\end{equation}
where the error matrix \(E\) collects the off-diagonal inner products \(\langle u_i,u_j\rangle\) from \eqref{eq:inner-product-bound} and any diagonal perturbation beyond \eqref{eq:ui-norm}.

Since \(U=Q\,R\) and \(R^\top R=U^\top U\), the diagonal entries obey
\begin{equation*}\label{eq:Rjj-squared}
R_{jj}^2
=
\frac{n}{m}+E_{jj}
=
\frac{n}{m}\bigl(1+\mathcal{O}(n^{-1})\bigr),
\end{equation*}
so 
\begin{equation}\label{eq:Rjj}
R_{jj}=\sqrt{\tfrac{n}{m}}\,(1+\mathcal{O}(n^{-1}))=\Theta(1).
\end{equation}

We now prove by induction on \(i\) that for each \(1\le i<j\le n\),
\(\;R_{ij}=\mathcal{O}(n^{-1})\).

Base case (\(i=1\)):  The Cholesky/QR recurrence gives
\begin{equation*}\label{eq:R1j-recurrence}
R_{1j}
=\frac{(R^\top R)_{1j}}{R_{11}}
=\frac{E_{1j}}{R_{11}}
=\frac{\mathcal{O}(n^{-1})}{\Theta(1)}
=\mathcal{O}(n^{-1}).
\end{equation*}
(Here we used \eqref{eq:gram-decomposition} to identify \((R^\top R)_{1j}=E_{1j}\) and \eqref{eq:Rjj} for the magnitude of \(R_{11}\).)

Inductive step:  Suppose for some \(i\ge2\) that
\(R_{\ell j}=\mathcal{O}(n^{-1})\) for all \(\ell<i\) and \(j>\ell\).  Then for each \(j>i\),
\begin{equation}\label{eq:Rij-recurrence}
R_{ij}
=\frac{1}{R_{ii}}\Bigl((R^\top R)_{ij}-\sum_{k=1}^{i-1}R_{k\,i}\,R_{k\,j}\Bigr).
\end{equation}
Here \((R^\top R)_{ij}=E_{ij}=\mathcal{O}(n^{-1})\) by \eqref{eq:gram-decomposition}, and each product
\(R_{k\,i}R_{k\,j}=\mathcal{O}(n^{-2})\) by the inductive hypothesis, so the sum over \(k=1,\dots,i-1\)
is \(\mathcal{O}(n^{-1})\).  Since \(R_{ii}=\Theta(1)\) by \eqref{eq:Rjj}, it follows from \eqref{eq:Rij-recurrence} that
\begin{equation}\label{eq:Rij-bound}
R_{ij}=\mathcal{O}(n^{-1}).
\end{equation}
We now obtain quantitative bounds on \(R^{-1}\).  From
\[
R = \sqrt{\tfrac{n}{m}}\,I_n + F,\qquad \|F\|_{\max}=\mathcal{O}(n^{-1}),
\]
and the fact that \(R\) is upper triangular with \(R_{jj}=\Theta(1)\), the inverse \(R^{-1}\) exists and is upper triangular.  In particular
\begin{equation}\label{eq:Rinv-diag-repl}
(R^{-1})_{jj}=\frac{1}{R_{jj}}
=\sqrt{\tfrac{m}{n}}\bigl(1+\mathcal{O}(n^{-1})\bigr).
\end{equation}

To control the off-diagonal entries, note that the columns of \(R^{-1}\) are the solutions \(x\) of \(R x=e_j\).  Equivalently, for each \(1\le i\le j\le n\),
\[
\sum_{k=i}^j R_{ik}(R^{-1})_{kj}=\delta_{ij},
\]
so when \(i<j\) we have the recurrence
\begin{equation}\label{eq:Rinv-recurrence-repl}
(R^{-1})_{ij}
=-\frac{1}{R_{ii}}\sum_{k=i+1}^j R_{ik}(R^{-1})_{kj}.
\end{equation}
We prove \((R^{-1})_{ij}=\mathcal{O}(n^{-1})\) by induction on the gap \(t=j-i\ge1\).  For \(t=1\) the right-hand side of \eqref{eq:Rinv-recurrence-repl} is
\(-R_{i,i+1}(R^{-1})_{i+1,i+1}/R_{ii}\), and since \(R_{i,i+1}=\mathcal{O}(n^{-1})\) and \((R^{-1})_{i+1,i+1}=\Theta(1)\) by \eqref{eq:Rinv-diag-repl}, we get \((R^{-1})_{i,i+1}=\mathcal{O}(n^{-1})\).  For larger \(t\), assume the claim for all smaller gaps; then every \((R^{-1})_{kj}\) with \(k>i\) is \(\mathcal{O}(n^{-1})\) by the induction hypothesis, and every \(R_{ik}=\mathcal{O}(n^{-1})\), so each summand \(R_{ik}(R^{-1})_{kj}=\mathcal{O}(n^{-2})\). Summing over \(k=i+1,\dots,j\) (at most \(n\) terms) yields \(\mathcal{O}(n^{-1})\); dividing by \(R_{ii}=\Theta(1)\) gives \((R^{-1})_{ij}=\mathcal{O}(n^{-1})\).  Thus
\begin{equation}\label{eq:Rinv-offdiag-repl}
(R^{-1})_{ij}=\mathcal{O}(n^{-1})\qquad(i<j).
\end{equation}

Finally, returning to \(Q=U R^{-1}\) and splitting the \(k=j\) term as before,
\[
Q_{ij}
=\frac{H_{ij}}{\sqrt m}\,(R^{-1})_{jj}
+\frac{1}{\sqrt m}\sum_{k\ne j}H_{ik}(R^{-1})_{kj}.
\]
By \eqref{eq:Rinv-diag-repl} the first term equals \(\dfrac{H_{ij}}{\sqrt n}\bigl(1+\mathcal{O}(n^{-1})\bigr)\).  The second term is bounded in absolute value by
\[
\frac{1}{\sqrt m}\sum_{k\ne j}\bigl|H_{ik}(R^{-1})_{kj}\bigr|
\le \frac{C}{\sqrt m}\sum_{k\ne j}\mathcal{O}(n^{-1})
= \mathcal{O}\bigl(m^{-1/2}\bigr)=\mathcal{O}\bigl(n^{-1/2}\bigr),
\]
using \(|H_{ik}|\le C\) and \eqref{eq:Rinv-offdiag-repl}.  Hence
\[
Q_{ij}=\frac{H_{ij}}{\sqrt n}+\mathcal{O}\bigl(n^{-1/2}\bigr),
\]
as required.
\end{proof}

\asym*
\begin{proof}
We now apply Theorem~\ref{thm:degenerate-CLT} to compute $\beta(\hat A)$ for the two constructions described above.  
Throughout, we fix the matrix $A$ and take the randomness only over the Rademacher vector $\eps=(\eps_1,\dots,\eps_n)$.
For any fixed $n\times n$ matrix $A$ whose rows have $\ell_2$--norm $\sqrt n$, define
\[
S_n \;=\; \frac1{\sqrt n}\,A\,\eps,
\qquad
(S_n)_j \;=\; \frac1{\sqrt n}\sum_{i=1}^n A_{j i}\,\eps_i.
\]
The covariance matrix of $S_n$ is
\[
\Sigma
\;=\;
\Cov(S_n)
\;=\;
\frac1n\,A\,A^\top,
\]
so that $\Sigma_{jj} = 1$ for all $j$, and for $j\neq k$,
\[
\Sigma_{jk}
\;=\;
\frac1n\sum_{i=1}^n A_{j i} A_{k i}
\;=\;
\bigl\langle \hat A_{j,*},\,\hat A_{k,*}\bigr\rangle,
\]
where $\hat A$ is $A$ with each row normalized to unit length.  The theorem thus applies with this covariance.
We now treat the two constructions separately.

\subsubsection*{{1. Normalized Random Sign matrix}}
Let $S\in\R^{n\times n}$ be a normalized random sign matrix.  By Lemma~\ref{lem:rand‐sign‐maxip}, with probability $1-o(1)$,
\[
\max_{j\neq k} \bigl|\Sigma_{jk}\bigr|
\;=\;
\max_{j\neq k}\bigl|\langle \hat S_{j,*},\hat S_{k,*}\rangle\bigr|
\;=\;
\mathcal{O}(\sqrt{\tfrac{\log n}{n}}).
\]
Consequently, with high probability,
\[
\det(\Sigma_{j,k})
\;=\;
1-\Sigma_{jk}^2
\;\ge\;
1-\mathcal{O}(\tfrac{\log n}{n})
\;=\;
1-o(1),
\qquad
\frac{\det(\Sigma_{j,k,\ell})}{\det(\Sigma_{j,k})}
\;\ge\;
1-o(1),
\]
so that we may take $\alpha^2=1-o(1)$ and $\beta^2=1-o(1)$.  Moreover, since every entry of $X_{ij}$ satisfies $|X_{ij}|=1$, we may fix a positive constant $B$ so that
\[
\mathbb{E}\bigl[e^{|X_{ij}|/B}\bigr]\;\leq\;2.
\]

Applying Theorem~\ref{thm:degenerate-CLT}, we therefore obtain
\[
\sup_{a,b\in\R^n}
\Bigl|
P(a\leq S_n\leq b) \;-\; P(a\leq Z \leq b)
\Bigr|
\;\leq\;
\frac{C\,B^3}{\alpha^2\beta^2\,\sqrt n}\,(\log n)^{6.5}
\;=\;
\mathcal{O}(n^{-1/2}(\log n)^{6.5}),
\]
where $Z\sim N(0,\Sigma)$. Since $n$ is fixed we may take $t = t(n)$ as some function of $n$ and have \( a = t\cdot\mathbf{1} , b = -t\cdot\mathbf{1} \in \mathbb{R}^n\). This gives us the specific bound
\[
\sup_{t\in\R}
\Bigl|
P(\|S_n\|_\infty >t) \;-\; P(\|Z\|_\infty >t)
\Bigr|
\;\leq\;
\mathcal{O}(n^{-1/2}(\log n)^{6.5}).
\]
We now use this distributional distance bound in combination with the the tail bounds on \(\|S_n\|_\infty\) and \(\|Z\|_\infty\) to get bounds on the difference in expectation. We start by splitting the integral version of the expectation at an arbitary parameter $q$ (to be decided later) to get
\begin{equation}
\begin{split}
\mathbb{E}\bigl[\|S_n\|_\infty\bigr]
&= \int_{0}^{\infty}\Pr\bigl(\|Z\|_\infty>t\bigr)\,dt
- \int_{0}^{q}\!\bigl[\Pr(\|Z\|_\infty>t)-\Pr(\|S_n\|_\infty>t)\bigr]\,dt \\
&\quad- \int_{q}^{\infty}\Pr\bigl(\|Z\|_\infty>t\bigr)\,dt
+ \int_{q}^{\infty}\Pr\bigl(\|S_n\|_\infty>t\bigr)\,dt.
\end{split}
\end{equation}
Which can be bounded above by our earlier distributional distance bound to get
\begin{equation}
\begin{split}
\mathbb{E}\bigl[\|S_n\|_\infty\bigr]
&\le \mathbb{E}\bigl[\|Z\|_\infty\bigr]
+ q\,\sup_{0 \leq t \leq q}\Bigl|\Pr(\|Z\|_\infty>t)-\Pr(\|S_n\|_\infty>t)\Bigr| \\
&\quad- \int_{q}^{\infty}\Pr\bigl(\|Z\|_\infty>t\bigr)\,dt
+ \int_{q}^{\infty}\Pr\bigl(\|S_n\|_\infty>t\bigr)\,dt.
\end{split}
\end{equation}

So we have
\[
\mathbb{E} \| S_n\|_\infty
\;\le\;
\mathbb{E}\| Z \|_\infty
\;+\;q\,\delta
\;-\;\int_{q}^{\infty}P\bigl(\|Z\|_\infty>t\bigr)\,dt
\;+\;\int_{q}^{\infty}P\bigl(\|S_n\|_\infty>t\bigr)\,dt,
\]
where
\[
\delta \;=\;\sup_{t\in\R}\Bigl|\,P(\|Z\|_\infty>t)-P(\|S_n\|_\infty>t)\Bigr|
\;=\;
\mathcal{O}(n^{-1/2}(\log n)^{6.5}).
\]
Since $n$ is fixed, we may choose
\[
q \;=\;\sqrt{4\,\log n}
\]
and bound the individual pieces.

\noindent\textbf{Sup--error term.}
\[
q\,\delta
\;=\;
\sqrt{4\log n}
\;\cdot\;
\mathcal{O}(n^{-1/2}(\log n)^{6.5})
\;=\;
\mathcal{O}(n^{-1/2}(\log n)^{7}).
\]

\noindent\textbf{Gaussian--tail integral.}
Since
\[
  P\bigl(\|Z\|_\infty > t\bigr)
  \;=\;
  P\Bigl(\bigcup_{i=1}^n\{|Z_i|>t\}\Bigr)
  \;\le\;
  \sum_{i=1}^n P\bigl(|Z_i|>t\bigr)
  \;=\;
  2n\bigl[1-\Phi(t)\bigr],
\]
we have
\[
  \int_{q}^{\infty}P\bigl(\|Z\|_\infty>t\bigr)\,dt
  \;\le\;
  \int_{q}^{\infty}\frac{2n}{\sqrt{2\pi}}\,
  \frac{e^{-t^2/2}}{t}\,dt.
\]
Observe that for \(t\ge q\),
\(\frac1t\le\frac1q\), so
\[
  \int_{q}^{\infty}\frac{e^{-t^2/2}}{t}\,dt
  \le
  \frac1q\int_{q}^{\infty}e^{-t^2/2}\,dt
  \le
  \frac1q\int_{q}^{\infty}t\,e^{-t^2/2}\,dt
  =
  \frac{e^{-q^2/2}}{q},
\]
where the last inequality followed from the fact that \(q>1\). Hence
\[
  \int_{q}^{\infty}P(\|Z\|_\infty>t)\,dt
  \;\le\;
  \frac{2n}{\sqrt{2\pi}}\;\frac{e^{-q^2/2}}{q}
  =
  \mathcal{O}(\frac{n}{q}e^{-q^2/2}).
\]
Substituting \(q=\sqrt{4\log n}\) gives
\[
  \frac{n}{q}e^{-q^2/2}
  =
  \frac{n}{\sqrt{4\log n}}\;n^{-4/2}
  =
  n^{-1}(\log n)^{-1/2}.
\]

\noindent\textbf{Sign--matrix tail integral.}
Assume each row of \(S\) is normalized by \(1/\sqrt n\), so that each coordinate can be written
\[
S_{n,i}=\frac{1}{\sqrt n}\sum_{j=1}^n \xi_{ij},\qquad \xi_{ij}\in\{\pm1\}.
\]
By Hoeffding's inequality,
\[
\Pr\big(|S_{n,i}|>t\big)\le 2\exp\big(-t^2/2\big)\qquad(t>0),
\]
and hence by a union bound
\[
\Pr\big(\|S_n\|_\infty>t\big)\le 2n\exp\big(-t^2/2\big).
\]
It follows that
\[
\int_q^\infty \Pr\big(\|S_n\|_\infty>t\big)\,dt
\le 2n\int_q^\infty e^{-t^2/2}\,dt
\le 2n\frac{e^{-q^2/2}}{q}.
\]
With the choice \(q=\sqrt{4\log n}\) we have \(e^{-q^2/2}=n^{-2}\), so
\[
\int_q^\infty \Pr\big(\|S_n\|_\infty>t\big)\,dt
= O\!\Big(\frac{1}{n\sqrt{\log n}}\Big).
\]

Combining these three bounds, we obtain
\[
\mathbb{E}\| S_n \|_\infty
=
\mathbb{E}\| Z \|_\infty
\;+\;
\mathcal{O}(n^{-1/2}(\log n)^{7}).
\]
Finally, by Lemma~\ref{lem:max-gauss},
\[
\mathbb{E}\|Z\|_\infty
<
\sqrt{2\log(2n)}
\;-\;
\frac{\log\log(2n)}{2\sqrt{2\log(2n)}}
\;+\;
\mathcal{O}((\log n)^{-1/2}),
\]
and Theorem \ref{thm:chatterjee-sudakov-fernique}, and Lemma \ref{lem:rand‐sign‐maxip}, combine to show us that the maximum deficit from the upper bound in this case is \(\mathcal{O}(\frac{(\log{n})^{3/4}}{\sqrt{n}})\) with high probability, so this inequality is tight in the leading order. The same expansion holds for \(\E\|S_n\|_\infty\), up to the additional 
\(\mathcal{O}(n^{-1/2}(\log n)^{7})\)–error.
Hence
\begin{equation}
\label{eq:beta-sign}
\boxed{
\beta(S)
\;=\;
\E\|S_n\|_\infty
\;=\;\sqrt{2\log(2n)}
\;-\;
\frac{\log\log(2n)}{2\sqrt{2\log(2n)}}
\;+\;
\mathcal{O}((\log n)^{-1/2}),
}
\end{equation}
with high probability.

\subsubsection*{2.~Orthogonal Almost--Hadamard matrix}
Let $Q\in\R^{n\times n}$ be an orthogonal almost--Hadamard matrix.  By definition, $Q Q^\top = I_n$.  Consequently,
\[
\Sigma
\;=\;
\frac1n\,Q\,Q^\top
\;=\;
\frac1n\,I_n,
\]
and upon scaling by $\sqrt n$ in the CLT normalization, we have $\Sigma=I_n$.  Thus
\[
\alpha^2 \;=\; 1,
\qquad
\beta^2 \;=\; 1,
\qquad
\]
\emph{exactly}, and as before we may take $B$ fixed.  Theorem~\ref{thm:degenerate-CLT} therefore gives
\[
\sup_{t\in\R}
\Bigl|
P(\|S_n\|_\infty >t) \;-\; P(\|Z\|_\infty >t)
\Bigr|
\;\leq\;
\mathcal{O}(n^{-1/2}(\log n)^{6.5}).
\]
where now $Z\sim N(0,I_n)$.  As above,
\[
\E\|Z\|_\infty
\;=\;
\sqrt{2\log n}
-\frac{\log\log n+\log(4\pi)}{2\sqrt{2\log n}}
+o\bigl((\log n)^{-1/2}\bigr),
\]
By Lemma~\ref{lem:trunc-hadamard-flat}, each entry of \(Q\) satisfies
\(\lvert Q_{ji}\rvert=\mathcal{O}(1/\sqrt{n})\), so we may fix a constant \[M = \max\limits_{i, j} Q_{ij}.\]
Continuing from the decomposition and the choice 
\[
q=\sqrt{2M\log n},
\]
the sup‐error term and the Gaussian‐tail integral remain unchanged in order.  It remains only to bound
\[
\int_{q}^{\infty}P\bigl(\|S_n\|_\infty>t\bigr)\,dt.
\]
  Hence
\[
(S_n)_j
=\sum_{i=1}^n Q_{ji}\,\eps_i
\]
is a sum of \(n\) independent, mean‐zero terms each bounded by \(\mathcal{O}(1/\sqrt{n})\). Hoeffding’s inequality gives us
\[
P\bigl(\lvert (S_n)_j\rvert>t\bigr)
\;\le\;
2\exp\!\Bigl(-\frac{2\,t^2}{\sum_{i=1}^n (Q_{ji})^2}\Bigr)
\;\le\;
2\exp\!\bigl(\frac{-2t^2}{M}\bigr).
\]
A union bound over \(j=1,\dots,n\) then yields
\[
P\bigl(\|S_n\|_\infty>t\bigr)
\;\le\;
2n\,\exp(\frac{-2t^2}{M}).
\]
Therefore
\[
\int_{q}^{\infty}P(\|S_n\|_\infty>t)\,dt
\;\le\;
2n\int_{q}^{\infty}e^{-2t^2/M}\,dt
\;=\;
\mathcal{O}(\frac{n}{q}e^{-2q^2/M})
\;=\;
\mathcal{O}(n^{-1}(\log n)^{-1/2})
\]
This completes the estimate of the matrix tail integral, which is still negligible compared to the CLT error and so
\begin{equation}
\label{eq:beta-hadamard}
\boxed{
\beta(Q)
\;=\;
\E\|Q\,\|_\infty
\;=\;
\sqrt{2\log(2n)}
\;-\;
\frac{\log\log(2n)}{2\sqrt{2\log(2n)}}
\;+\;
\mathcal{O}((\log n)^{-1/2}),
}
\end{equation}
if we are a constant independent of $n$ away from the closest Hadamard matrix of order $\ge n$.
Combining~\eqref{eq:beta-sign} and~\eqref{eq:beta-hadamard}, we conclude that for both constructions,
\[
\boxed{
\beta(\hat A)
\;=\;
\sqrt{2\log (2n)}
-\frac{\log\log (2n)}{2\sqrt{2\log (2n)}}
+\mathcal{O}((\log n)^{-1/2}).
}
\]

In fact, we have the chain of inequalities:
\[
\boxed{
\beta(S)
\;<\;\mathbb{E}\bigl\|\!Z\!\bigr\|_\infty
+\mathcal{O}(n^{-1/2}(\log n)^{7})
\;\le\;\beta(Q)
+\mathcal{O}(n^{-1/2}(\log n)^{7}).
}
\]

So we see that they are within some \(\polylog(n)/\sqrt{n}\) factor of each other. Although from the covariance matrix heuristic, we can see that the Orthonormal Almost-Hadamard matrix should do better, there is no way to explicitly prove this without some stronger results or conditions that we are unaware of. A quantitative lower bound to complement Chatterjee's upper bound would work, but to our knowledge, such a result is not known.
\end{proof}

From the `integral bounding' type argument, we can see that under the assumptions of boundedness and non-degeneracy, the $\beta$-rate attained by these families is the optimal asymptotic expansion for the global maximizer as well.

\begin{restatable}[Asymptotic optimality of the leading expansion]{lemma}{optRate}\label{lem:optimality}
Let $(A^{(n)})_{n\ge 1}$ be a sequence of real $n\times n$ matrices with each row of $A^{(n)}$ having Euclidean norm $\sqrt n$.  Assume the following uniform conditions hold as $n\to\infty$:
\begin{enumerate}
  \item (Bounded entries) there exists $M>0$ such that $\lvert A^{(n)}_{ij}\rvert\le M$ for all $n,i,j$;
  \item (Non–degeneracy / CLT hypotheses) for the covariance matrices
    \[
      \Sigma^{(n)} \;=\; \frac{1}{n}A^{(n)}(A^{(n)})^\top
    \]
    there exist constants $\alpha,\beta>0$ (independent of $n$) with
    \[
      \det\bigl(\Sigma^{(n)}_{j,k}\bigr)>\alpha^2,\qquad
      \frac{\det\bigl(\Sigma^{(n)}_{j,k,\ell}\bigr)}{\det\bigl(\Sigma^{(n)}_{j,k}\bigr)}>\beta^2
    \]
    for all distinct indices $j,k,\ell$.
\end{enumerate}
Let $\widehat A^{(n)}=A^{(n)}/\sqrt n$ then
\[
  \beta\bigl(\widehat A^{(n)}\bigr)
  \;=\;
  \mathbb{E}\bigl\|Z^{(n)}\bigr\|_\infty
  \;+\; o((\log n)^{-1/2}),
\]
where $Z^{(n)}\sim\mathcal N(0,\Sigma^{(n)})$.  Consequently, by Lemma~\ref{lem:max-gauss} (Gaussian maximum asymptotics), one has the universal expansion
\begin{equation}\label{eq:optimal-expansion}
\beta\bigl(\widehat A^{(n)}\bigr)
<
\sqrt{2\log(2n)}
\;-\;
\frac{\log\log(2n)+\log(4\pi)}{2\,\sqrt{2\log(2n)}}
\;+\;
\frac{\gamma}{\sqrt{2\log(2n)}}
\;+\;
o\bigl((\log n)^{-1/2}\bigr),
\end{equation}
and the inequality in \eqref{eq:optimal-expansion} is asymptotically tight under the stated hypotheses.
\end{restatable}

\begin{proof}
Fix $n$ and suppress the superscript $(n)$ for clarity.  With the hypotheses above the high–dimensional CLT (Theorem~\ref{thm:degenerate-CLT}) applies: because the pre–scaled entries are uniformly bounded by $M$, there exists a fixed $B>0$ (depending only on $M$) so that the exponential moment assumption of Theorem~\ref{thm:degenerate-CLT} holds, and by the uniform non–degeneracy constants $\alpha,\beta>0$ the theorem yields the uniform Kolmogorov–type bound
\[
  \delta_n := \sup_{a,b\in\mathbb R^n}\bigl|\,\Pr(a\le S_n\le b)-\Pr(a\le Z_n\le b)\bigr|
  \;=\;
  \mathcal O\!\bigl(n^{-1/2}(\log n)^{6.5}\bigr).
\]
(Here $S_n=\widehat A\eps$ and $Z_n\sim\mathcal N(0,\Sigma)$ as in the statement.)
Choosing
\[
q=\sqrt{2M\log n},
\]
and following the exact integral bounding argument from the proof of Theorem \ref{thm:asymptotics} yields
\[
  \mathbb{E}\|S_n\|_\infty
  =\mathbb{E}\|Z_n\|_\infty
  +o((\log n)^{-1/2}).
\]
Finally, Lemma~\ref{lem:max-gauss} (the Gaussian maximum asymptotic expansion) gives the precise two–term asymptotic for $\mathbb{E}\|Z_n\|_\infty$:
\[
  \mathbb{E}\|Z_n\|_\infty
  <
  \sqrt{2\log(2n)}
  -\frac{\log\log(2n)}{2\sqrt{2\log(2n)}}
  +o((\log n)^{-1/2}),
\]
and substituting this into the previous equality proves \eqref{eq:optimal-expansion}.  This shows that no sequence of matrices satisfying the stated uniform hypotheses can improve the first two terms of the expansion; and the rate achieved by our Orthonormal Almost-Hadamard matrices show that this is tight.
\end{proof}

\section{Interpretation of earlier constructions}

We use the structural results proved in this paper to address an example that resolves part of the central question in the existing papers: Why are the known optimal low-dimensional matrices so highly structured, but their natural generalization becomes suboptimal for large dimensions, and random matrices without this nice structure take over? 

We first introduce a well-known family of subsets of \(\{-1,1\}^n\).
\begin{definition}[Subcubes]\label{def:subcubes}
Let \(n\in\mathbb{N}\) and \(k\in\{0,1,\dots,n\}\). Choose an index set \(S\subseteq\{1,\dots,n\}\) with \(|S|=k\) and a fixed assignment \(a\in\{-1,1\}^S\). The subcube of co-dimension \(k\) determined by \((S,a)\) is
\[
C_{S,a}
=\bigl\{\,x\in\{-1,1\}^n : x_i = a_i\ \forall\,i\in S\bigr\}.
\]
\end{definition}

The paper \cite{ExplicitConstructions} presents the following construction that has \(\beta(\cdot)\) \(18\%\) smaller than the optimal rate:
\begin{definition}[Highly balanced binary trees]\label{def:balanced-trees}
    For a fixed integer $n \geq 1$, fill up a binary tree with vertices from left to right until one has $n$ leaves, and finally add an edge that points into the root.
\end{definition}
We label the edges of such a highly balanced binary tree in the following way: edges that point left have label $-1$,  edges that point right have label $1$, and the edge that points to the root has label $1$.  From here, for a leaf $v$, walk along the unique path from the root to $v$. Then the edge labels of this path becomes a row of the matrix this method generates, where if the length of the path is less than $n$, we make the rest of the entries 0. The case $n=4$ is illustrated.

\begin{figure}[h!]
    \centering
    \begin{minipage}{.45\textwidth}
        \centering
    \[
        \begin{tikzcd}
        	&& {} \\
        	&& \bullet \\
        	& \bullet && \bullet \\
        	\bullet & \bullet && \bullet & \bullet
        	\arrow["1"', from=1-3, to=2-3]
        	\arrow["{-1}", from=2-3, to=3-2]
        	\arrow["1"', from=2-3, to=3-4]
        	\arrow["{-1}", from=3-2, to=4-1]
        	\arrow["1", from=3-2, to=4-2]
        	\arrow["{-1}"', from=3-4, to=4-4]
        	\arrow["1"', from=3-4, to=4-5]
        \end{tikzcd}\]
    \end{minipage}%
    \begin{minipage}{0.5\textwidth}
        \centering
\begin{align*}
    \left[\begin{array}{rrrr}
        1 & -1 & -1 & 0 \\
        1 & -1 & 1 & 0 \\
        1 & 1 & -1 & 0 \\
        1 & 1 & 1 & 0
    \end{array}\right]
\end{align*}
    \end{minipage}
    \caption{An unsatisfiable tree and the corresponding matrix.}
\end{figure}
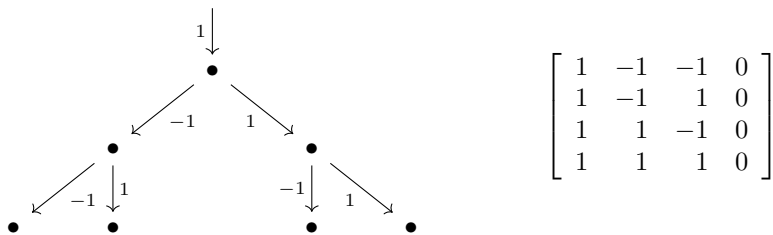

It is easy to see then that the root to leaf paths specify the fixed co-ordinates of a \(n\)-way subcube partition of the half of the hypercube \(\{-1,1\}^n\) with first coordinate +1, by subcubes of co-dimension \(\floor{\log_{2}(n)}\) and \(\floor{\log_{2}(n)+1}\). Figure \ref{fig:subcubes} provides a picture of these subcubes for $n=4$. Since we leave all the remaining co-ordinates as 0, each row's Voronoi cell is the subcube specified by the fixed co-ordinates. Since the centroid of vectors in a subcube is trivially the vector with the fixed coordinates as they are and 0s everywhere else, this is indeed a centroidal Voronoi tessellation. The balanced nature of the tree also guarantees cell sizes within a factor of 2 of each other (not quite asymptotically equal, but in the case that n is a power of 2, we have exactly equal sizes). By our characterization of the optimum earlier, we see that the only thing holding us back is the level-1 inequality. Indeed, intuitively speaking, packing a cube with smaller cubes is not isoperimetrically optimal, something we now make precise. 
\begin{lemma}[Subcubes are Suboptimal]\label{lem:Subcubes are Suboptimal}
Let \( T_i \) be a \(\floor{\log_2(n) +1}\) co-dimensional subcube of \( \{-1,1\}^n \). Then, \[
\sqrt{W_1[\mathbf{1_{T_i}}]}  = \frac{\sqrt{\floor{\log_2(n) +1}}}{n}
\]
\end{lemma}

\begin{proof}
The Level-1 weight can be rewritten as
\[
\sqrt{\sum_{j=1}^n \widehat{\mathbf{1}_{S_i}}(\{j\})^2}\]

On the fixed coordinates, the Fourier weights are simply the volume of \(S_i\), and on the free coordinates, we have 
$
\hat{\mathbf{1_{S_i}}}(x)= 0.
$
Since \(\floor{\log_2(n)+1}\) of the coordinates are fixed, we get the required expression.
$
\sqrt{\floor{\log_2(n)+1}}/n.
$

\end{proof}

This shows that the subcube decomposition in the best case ($n$ is a power of 2) yields \(
\beta(A) = \sqrt{\log_2(n)+1}\), which is the same as the best derandomized construction found in the existing literature.
One may then easily check that the explicit matrices in \cite{Steinerberger} arise exactly from subcube partitions of \(\{-1,1\}^n\) and thus have asymptotically suboptimal generalizations.

\section*{Acknowledgments}
I am grateful to Stefan Steinerberger for his continued guidance, feedback, and helpful conversations throughout this project. I also thank Alex Albors, Hisham Bhatti, Lukshya Ganjoo, Raymond Guo, Dmitriy Kunisky, Rohan Mukherjee, Alicia Stepin, and Tony Zeng for generously sharing their results and for valuable discussions related to the structure of Bad Science matrices.
\section*{Data and code availability}
All experiments and code used to generate figures are available at the author's GitHub repository: 
\href{https://github.com/ShridharS19/Bad-Science-Experiments}
{https://github.com/\allowbreak ShridharS19/\allowbreak Bad-\allowbreak Science-\allowbreak Experiments}

\end{document}